\documentclass{daj}
\usepackage{bbm,amssymb,amsmath,amsthm}

\newtheorem{thm}{Theorem}[section]
\newtheorem{lem}[thm]{Lemma}

\newtheorem{cor}[thm]{Corollary}

\DeclareMathOperator{\Bin}{Bin}  
\newcommand{\NN}{{\mathbb N}}
\newcommand{\er}{\mathrm{e}}

\newcommand{\cC}{\ensuremath{\mathcal{C}}}

\newcommand{\cE}{\ensuremath{\mathcal{E}}}

\newcommand{\cH}{\ensuremath{\mathcal{H}}}

\newcommand{\cJ}{\ensuremath{\mathcal{J}}}

\newcommand{\cS}{\ensuremath{\mathcal{S}}}

\newcommand{\cT}{\ensuremath{\mathcal{T}}}

\newcommand{\cZ}{\ensuremath{\mathcal{Z}}}

\newcommand{\PR}{\mathbb P}
\newcommand{\E}{{\mathbb E}\,}

\newcommand{\be}{\begin{equation}}
\newcommand{\ee}{\end{equation}}

\newcommand{\ssum}[1]{\sum_{\substack{#1}}}  
\newcommand{\sprod}[1]{\prod_{\substack{#1}}}  

\newcommand{\lam}{\ensuremath{\lambda}}
\renewcommand{\a}{\ensuremath{\alpha}}
\renewcommand{\b}{\ensuremath{\beta}}
\newcommand{\eps}{\ensuremath{\varepsilon}}

\renewcommand{\le}{\leqslant}

\renewcommand{\ge}{\geqslant}
\renewcommand{\geq}{\geqslant}

\newcommand{\fl}[1]{{\ensuremath{\left\lfloor {#1} \right\rfloor}}}
\newcommand{\cl}[1]{{\ensuremath{\left\lceil #1 \right\rceil}}}
\renewcommand{\(}{\left(}
\renewcommand{\)}{\right)}
\newcommand{\pfrac}[2]{\left(\frac{#1}{#2}\right)}  
\newcommand{\hh}{\ensuremath{\mathbf{h}}}

\newcommand{\one}{\ensuremath{\mathbbm{1}}} 

\dajAUTHORdetails{
  title = {Cycle Type of Random Permutations:\\ a Toolkit},
  author = {Kevin Ford},
  plaintextauthor = {Kevin Ford},
  keywords = {random permutations, cycle type, },
}  

\dajEDITORdetails{%
   year={2022},
   number={9},
   received={10 May 2021},   
   revised={18 February 2022},    
   published={8 September 2022},  
   doi={10.19086/da.38090},       
}   

\begin{document}

\begin{frontmatter}[classification=text]

  \author[kbf]{Kevin Ford\thanks{Supported by National Science Foundation
      grant DMS-1802139.}}

\begin{abstract}
We provide a standard reference for fundamental distributional results
about the cycle type of a random permutation $\sigma \in \cS_n$, emphasizing methods
which are combinatorial or probabilistic in nature and adaptable to other
situations.  Many of our techniques are borrowed from methods used to
prove analogous theorems about the prime factorization of random integers.
Included here are results about the proportion of permutations $\sigma$ having
a given number of cycles with lengths from a given set, the distribution of
the smallest and largest cycle, and the distribution of the sizes of fixed sets
of $\sigma$.

\end{abstract}
\end{frontmatter}

\section{Introduction}

The theory of the cycle type of
random permutations of the symmetric group $\cS_n$
is very active, with many applications in combinatorics, group theory
and number theory.  A selection of applications includes
\begin{itemize}
\item the distribution of orders of permutations
(the least common multiple of cycle lengths)
\cite{equalorders, BGHP, BLGNPS, CHM, ET1, ET2, ET3, ET4, ET5, ET6, ET7, GS,Landau-perm,Massias,Nicolas,NP,NP2} and \cite[Sec. 6]{granville};
\item invariable generation of the symmetric group
\cite{dixon92, EFG2, LP, PPR} and other classical 
groups \cite{McK};
\item the distribution of fixed sets (divisors)
of permutations \cite{dfg08, EFG1, EFG2, EFK, FGK, LP, weingartner};
\item permutations contained in transitive subgroups
\cite{camkan, EFK, Jordan};
\item irreducibility of polynomials over the rationals
\cite{BSK, BSKK};
\item permutation groups containing elements with a single
cycle that is not a fixed point (Jordan groups) \cite{Jordan, GPU}
and \cite[Ch. 10]{Seress};
\item polynomial factorization in finite fields
\cite{ABT, BSK, Rudnick19}.
\end{itemize}

The main purpose of this paper is provide a standard reference for fundamental distributional results
about cycle types, which heretofore have been scattered across many papers
with widely varying strength and generality.
We showcase methods which are both
\emph{general} and \emph{combinatorial}.    
While many of the results stated here are weaker than existing results
in the literature, they are far more general,
have significantly shorter proofs and are more adaptable to new situations.
This paper is an expanded version of portions of the author's lecture notes 
on permutations prepared for the course ``Anatomy of integers and random permutations''.

Our methods are borrowed from the
theory of numbers, particularly the theory of sieves
and the theory of averages of multiplicative functions
(see \cite[Part 3, Part 4]{Koubook} for uses in number theory).
As positive integers factor uniquely into a product of prime numbers,
and permutations factor uniquely into a product of cycles,
the connection between the distributions of the two objects,
prime factors and cycles, is not surprising.
The first explicit mention of such a connection, however,
is the paper of Knuth and Trabb Pardo \cite{KTP} in 1976.
On the other hand, there are significant differences in the structure of the two objects which explains why there is no simple \emph{transference principle}
between statements about prime factorizations and the corresponding statement 
about the cycle structure of permutations.
 Deeper inspection, however, reveals that
the \emph{distribution} of the two factorizations have many common features,
and for much the same underlying reasons.

Let $\sigma$ denote a random permutation from
the symmetric group $\cS_n$, each permutation being equally
likely\footnote{Random permutations sampled
from certain other distributions have been studied, e.g. \cite{ABTbook}, but we will not discuss these here.}.
We denote by $\PR_n$ and $\E_n$ the probability and expectation
with respect to a uniform random $\sigma\in \cS_n$.
Often, the subscript $n$ will be omitted if it is clear from the context.
We denote the type (or cycle type) of $\sigma$ by
\[
(C_1(\sigma), C_2(\sigma),\ldots, C_n(\sigma)),
\]
where
$C_j(\sigma)$ is the number of cycles of length $j$ in $\sigma$.
More generally, for any subset $I$ of $[n]=\{1,\ldots,n\}$, we let $C_I(\sigma)$
be the number of cycles whose lengths lie in the set $I$.
For brevity, we write $C(\sigma)$ for the total number of
cycles in $\sigma$.
The principal problems considered in this paper are
\begin{itemize}
\item[(a)] What is the distribution of $C_j(\sigma)$ for each $j$?
\item[(b)] What is the distribution of $C_I(\sigma)$ for each $I$?
\item[(c)] What is the joint distribution of $C_{I_1}(\sigma),
\ldots,C_{I_k}(\sigma)$ for disjoint sets $I_1,\ldots,I_k\subseteq [n]$?
\item[(d)] What is the distribution of $C_I(\sigma)$
conditional on $C(\sigma)=k$?
\end{itemize}

Most of the analysis of these problems in the literature utilizes
recurrence relations, properties of Stirling numbers, or 
complex analytic methods
starting with the exponential generating function of Gruder
\cite[Satz 2]{Gruder} for permutations having only cycle sizes from a set $I$.
See, e.g. \cite{FlSe} for a general analytic theory.

\begin{thm}[Gruder]\label{thm:genfcn}
For complex $x$ and $y$ with $|x|<1$, and subset $I\subseteq \NN$ we have
\be\label{genfcn}
\sum_k \sum_n \PR_n \big(C_I(\sigma)=k,C_{[n]\setminus I}(\sigma)=0\big) x^n y^k=
\exp \bigg\{ y \sum_{m\in I} \frac{x^m}{m}
\bigg\}.
\ee
Moreover, when $I$ is finite the above identity holds for every complex $x$.
\end{thm}

While some existing distribution theorems are very strong,
in particular the recent results
of Manstavi\v{c}ius and Petuchovas
 \cite{ManPet16, ManPet17, Pet16, Pet18},
the methods are highly specialized and not easily adaptable to the solution of related problems.
By contrast, 
we eschew recurrences and generating functions
(for the most part) in favor of direct arguments.
We focus on
\emph{quantitative} results, that is, with a specific rate
of convergence, as well as results that are \emph{uniform}
in $j,I$ and the sets $I_j$.  

Underlying our analysis is the
\emph{Poisson model} of permutations, which suggests that
$C_j(\sigma)$ is approximately Poisson with parameter $1/j$,
and that $C_1(\sigma), C_2(\sigma), \ldots$ are nearly independent.
This is already hinted at in Cauchy's classical formula:

\begin{thm}[Cauchy]\label{thm:Cauchy}
If $m_1+2m_2+\cdots+nm_n=n$, then
\[
\PR_n \big( C_1(\sigma)=m_1,\ldots,C_n(\sigma)=m_n \big) = \prod_{j=1}^n \frac{(1/j)^{m_j}}{m_j!}.
\]
\end{thm}

If  $X_1, X_2, \ldots, X_k$ are independent Poisson random variables 
with parameters $\lambda_1,\ldots,\lambda_k$, respectively, then the sum
$X_1+\cdots +X_k$ is Poisson with parameter $\lambda_1+\cdots+\lambda_k$.
Thus, for subsets $I$ of $[n]$ we should expect that $C_I(\sigma)$ will be roughly 
Poisson with parameter 
\[
H(I) := \sum_{j\in I} \frac{1}{j}.
\]
In the important special case $I=\{1,\ldots,n\}$ we set 
\[
H_n = \sum_{i=1}^n \frac{1}{i}.
\]
The Poisson model has limitations, however,
particularly if $I$ contains many large elements.
For example the events
``$C_j(\sigma)\ge 1$'', $n/2<j\le n$, are clearly disjoint.
Also, if $I=\{2,\ldots,n\}$, then $\PR (C_I(\sigma)=0) = 1/n!$
whereas the Poisson model predicts a probability
of about $\er^{-H(I)} \approx 1/n$.
In general, permutations lacking large cycles are much rarer than would be 
predicted by the Poisson model, these being analogous to integers
lacking large prime factors.  We will take up this subject again later,
e.g. Theorem \ref{nolargecycles}.
On the other hand, we shall see that the Poisson model is very accurate for
small $j$, and is reasonably accurate for large $j$ on average
near the center of the distribution.

In the remainder of the introductory section, we describe a number of results, most of which will be proved in subsequent sections.

\subsection{Notational conventions.}
We adopt the standard Bachman-Landau, Hardy, and Vinogradov notations:
$f=O(g)$ and $f\ll g$ mean that there is a positive constant $C$
so that $|f| \le Cg$ throughout the domain of $f$.  The constant $C$ is independent of any parameters, unless specified by subscripts, e.g.
$f(x)=O_{\eps} (x^\eps)$.
Also, $f(x) \sim g(x)$ as $x\to \infty$ means $\lim_{x\to \infty} f(x)/g(x)=1$
and $f(x)=o(g(x))$ means that $\lim_{x\to\infty} f(x)/g(x)=0$.

For $\sigma\in \cS_n$, the notation
$\beta|\sigma$ means that $\beta$ is a \emph{divisor} of the permutation
$\sigma$, i.e. a product of some subset of the cycles of $\sigma$.
$|\beta|$ is the size (length) of $\beta$.

$\one(S)$ is the indicator function of statement $S$; $\displaystyle \one(S) = 
\begin{cases} 1 & S \text{ is true} \\ 0 & S \text{ is false}. \end{cases}$

\subsection{Binomial moments.}
A great deal of our analysis ultimately relies on
estimates for joint binomial moments of the quantities
$C_I(\sigma)$.  Recall that if $X$
is Poisson with parameter $\lambda$,
then for any non-negative integer $m$,
\[
\E \binom{X}{m} = 
\sum_{k=0}^\infty \binom{k}{m} \er^{-\lambda} \frac{\lambda^k}{k!}= \frac{\lambda^m}{m!}.
\]
We establish an analog for joint binomial moments of the statistics
$C_{I_j}(\sigma)$ for disjoint $I_1,\ldots,I_k$.

\begin{thm}\label{cycles_sets_expectation}
Let $I_1,I_2, \ldots, I_k$ be disjoint, nonempty subsets of $[n]$, and let
$m_1,\ldots,m_k$ be non-negative integers.  Then
\[
\E \binom{C_{I_1}(\sigma)}{m_1} \cdots \binom{C_{I_k}(\sigma)}{m_k} \le \prod_{j=1}^k \frac{H(I_j)^{m_j}}{m_j!},
\]
with equality if and only if $\sum_{j=1}^k m_j \max(I_j) \le n$.
\end{thm}

In the special case $k=1$, $I_1=[m]$ and $m_1=1$ we have
\be\label{exp-cyclestom}
\E C_{[m]}(\sigma) = H_m = \log m + \gamma + O(1/m).
\ee

Theorem \ref{cycles_sets_expectation} will be proved in Section 
\ref{sec:binom}, where we will also give short deductions of
Theorems \ref{thm:genfcn} and \ref{thm:Cauchy} from
Theorem \ref{cycles_sets_expectation}.

\subsection{Local limit theorems}
We begin with an exact evaluation of the local limit laws for $C_j(\sigma)$, due to
Goncharov \cite{Gon44}.

\begin{thm}[Goncharov]\label{thm:Goncharov}
For any $n\in \NN$, $1\le j\le n$ and $0\le m\le n/j$, we have
\[
\PR \( C_j(\sigma)=m \) = \frac{(1/j)^m}{m!} \sum_{h=0}^{\fl{n/j}-m}  \frac{(-1/j)^h}{h!}, \quad (1\le j\le n, 0\le m\le n/j).
\]
\end{thm}

A special case is the
very classical \emph{derangement problem}, posed in 1708 by  Pierre Raymond de Montmort.
Taking $j=1$ we have the exact formula for derangements
\[
\PR(C_1(\sigma)=0 ) = \sum_{j=0}^n \frac{(-1)^j}{j!}.
\]
Observe that if $j,m$ vary with $n$ such that 
$mj\le n$ and that
either $j\to \infty$ or $\frac{n}{j}-m\to \infty$ then
\[
\lim_{n\to \infty} \frac{\PR_n \( C_j(\sigma)=m \)}{\er^{-1/j} (1/j)^m/{m!}} =1.
\]
This establishes the Poisson distribution of $C_j(\sigma)$ in this range.

Theorem \ref{thm:Goncharov} can be thought of as a permutation analog of Landau's \cite[p. 211]{Landau} classical theorem in number theory, which states that
number of integers $n\le x$ 
having exactly $k$ distinct prime factors is asymptotic to
\[
\frac{x}{\log x}\; \frac{(\log\log x)^{k-1}}{(k-1)!}
\]
as $x\to\infty$.

Here we derive a very general local limit law.  In such generality,
we only obtain an upper bound for the probability
 of the expected order.  Lower bounds
are also possible, as are asymptotic formulae, when 
working with small cycle lengths; see Theorem \ref{Poisson_smallcycles} below.
The behavior of $\PR(C_I(\sigma)=0)$ when $I=\{m+1,\ldots,n\}$
is very different from the Poisson model prediction
and will be dealt with separately.

\begin{thm}\label{cycles_sets}
Let $I_1,\ldots,I_r$ be arbitrary disjoint, nonempty subsets of $[n]$ and $m_1,\ldots,m_r\ge 0$.  Then
\[
\PR\big( C_{I_1}(\sigma)=m_1,\ldots,C_{I_r}(\sigma)=m_r \big) \le \frac{\er^{H_n}}{n} \, \prod_{j=1}^r \Bigg( \frac{H(I_j)^{m_j}}{m_j!} \er^{-H(I_j)}\Bigg) \cdot \(\eps+\frac{m_1}{H(I_1)}+\cdots+\frac{m_r}{H(I_r)}\),
\]
where $\eps=0$ if $[n]=I_1\cup \cdots \cup I_r$ and $\eps=1$
otherwise.
\end{thm}

The analog of Theorem \ref{cycles_sets} for prime factors of a random integer $n\le x$ was proved by the author \cite{primepoisson}.  We note that $H_n \le \log n + 1$
for all $n$, thus the factor $\er^{H_n}/n$ is  bounded.
Consequently, whenever $r$ is bounded and $m_j = O(H(I_j))$ for each $j$,
the right side is 
\[
O\big( \PR (Y_1=m_1,\ldots,Y_r=m_r) \big),
\] 
where for each $i$, $Y_i$ is Poisson with parameter $H(I_i)$,
and $Y_1,\ldots,Y_r$ are independent.
Thus, Theorem \ref{cycles_sets} gives an upper bound for counts of 
cycle lengths in sets $I_1,\ldots,I_r$ of the expected order (up to 
a constant factor) according to the Poisson model.   As a special case of one set $I_1=I$, we
obtain:

\begin{cor}\label{cycles-single-set}
For any $I\subset [n]$ and $m\ge 0$,
\[
\PR\( C_{I}(\sigma)=m \) \le \frac{\er^{H_n-H(I)}}{n} \cdot \frac{H(I)^m}{m!}
\Bigg( \one(I\ne [n]) + \frac{m}{H(I)} \Bigg).
\]
In particular,
\[
\PR( C_I(\sigma)=0 ) \le \frac{\er^{H_n-H(I)}}{n} = \er^{\gamma-H(I)}(1+O(1/n)).
\]
\end{cor}

The first estimate is asymptotically sharp in the case $I=[n]$
and $m=o(\log n)$ as $n\to \infty$; see \eqref{lower-smallk} below 
for a corresponding lower bound.

A slight improvement of the final estimate, namely
$\PR(C_I(\sigma)=0) \le \er^{\gamma-H(I)}$,
is given in \cite{GPU} using different methods.

Theorem \ref{cycles_sets} becomes less accurate
when $m_j$ is much larger than $H(I_j)$, however it still gives
roughly the right rate of decay; e.g. when $I=[n]$ and $k=n$,
$\PR(C(\sigma)=n)=1/n!$ while the right side is
$O(H_n^{n-1}/n!)$.

Corollary \ref{cycles-single-set} is a permutation analog of 
the Hardy-Ramanujan \cite{HR17} inequality
\[
\# \{n\le x:n\text{ has exactly }k\text{ distinct prime factors} \} \le C_1 \frac{x}{\log x}\; \frac{(\log\log x+C_2)^{k-1}}{(k-1)!},
\]
where $C_1,C_2$ are certain absolute constants.

Theorem \ref{cycles_sets} is a useful tool for showing
that cycle counts cannot vary too much from their means.  Specifically, the local statistics obey
the same tail bounds as the Poisson distribution,
cf. Lemma \ref{Poisson_tails}.

\begin{thm}\label{single-set-tails}
Let $I$ be a nonempty subset of $[n]$. 
For $0\le \lambda \le 1$ we have
\[
\PR\big( C_I(\sigma) \le \lambda H(I) \big) \le 2 \er^{1-Q(\lambda)H(I)},
\]
where 
\[
Q(\lambda)=\lambda \log\lambda - \lambda+1 \ge 0.
\]
For $\lambda \ge 1$ we have
\[
\PR\big( C_I(\sigma) \ge \lambda H(I)+1 \big) \le 2 \er^{1-Q(\lambda)H(I)}.
\]
Lastly, when $0\le \psi \le \sqrt{H(I)}$,
\[
 \PR \( |C_{I}(\sigma) - H(I)| \ge  \psi \sqrt{H(I)} \)
\le 20 \er^{-\frac13 \psi^2}. 
 \]
\end{thm}

The function $Q$ is non-negative and satisfies $Q(x) \approx \frac12 (x-1)^2$
for $x$ near $1$; see also the inequality \eqref{Qx_crude} below.

When $\lambda$ is close to 1, we can be much more precise,
showing a Central Limit Theorem for $C_I(\sigma)$; see
Theorem \ref{CLT_cycles} below.

Specializing to cycle lengths in a single interval $I=[a,b]\cap \NN$,
and using that $H(I) \approx \log(b/a)$,
we obtain the following very useful estimates.

\begin{thm}\label{cyclesinterval}
Let $a,b$ be real numbers with
 $1\le a<b\le n$ and set $I=[a,b]\cap \NN$.
Uniformly for $0 \le \lam \le 1$, we have
\[
\PR \( C_{I}(\sigma) \le \lam \log (b/a) \) = O\( (b/a)^{-Q(\lam)} \).
\] 
Let $\lambda_0>1$.  Uniformly for $1 \le \lam \le \lam_0$, 
\[
\PR \( C_{I}(\sigma) \ge \lam \log (b/a) \) =O_{\lambda_0} \((b/a)^{-Q(\lam)}\).
\]
In particular, uniformly for
$0\le \psi \le \sqrt{\log(b/a)}$, 
\[
 \PR \( |C_{I}(\sigma) - \log (b/a)| \ge  \psi \sqrt{\log (b/a)} \)=O\big(
 \er^{-\frac13 \psi^2} \big).
\]
\end{thm}

In particular, taking $I=[n]$, we see that
 $C(\sigma)$ usually does not vary more that a constant times $\sqrt{\log n}$ from its mean $H_n$.

Theorems \ref{cycles-single-set} and \ref{single-set-tails}
are not very accurate when $H(I)<1$, especially in the case $m=1$.
In this case, we expect that $C_I(\sigma)$ will rarely be much more than 1.  The next Theorem gives an improved upper bound in this case.

\begin{thm}\label{cycles-set-strict}
 If $I$ is a nonempty susbet of $[n]$, and $k\ge 0$,
then
\[
\PR \( C_I(\sigma) \ge k \) \le \frac{H(I)^k}{k!}.
\]
\end{thm}

The proof is very short and we include it here.  
By Theorem \ref{cycles_sets_expectation},
\[
\PR ( C_I(\sigma) \ge k ) \le \E \binom{C_I(\sigma)}{k} \le \frac{H(I)^k}{k!}.\qedhere
\]

\begin{cor}\label{two-equal-large-cycles}
Let $2\le \ell \le n$.
The probability that a random
permutation $\sigma\in \cS_n$ has two cycles of the same length $j$
for some $j\ge \ell$, is at most $\frac{1}{2(\ell-1)}$.
\end{cor}

Again, the proof is very short:
By Theorem \ref{cycles-set-strict}, $\PR( C_j(\sigma)\ge 2) \le \frac{1}{2j^2}$.  Summing over $j\ge \ell$ we find that
\[
\PR (C_j(\sigma)\ge 2\text{ for some }j\ge \ell ) \le \sum_{j=\ell}^\infty
\frac{1}{2j^2} \le \frac12 \sum_{j=\ell}^\infty \frac{1}{j(j-1)} = \frac{1}{2(\ell-1)}.
\]

Next, we take a first look at the \emph{random sequence}
$C_{[m]}(\sigma)$ $(1\le m\le n)$ for 
$\sigma\in \cS_n$.
As long as $m$ is not too small,
it is relatively easy to deduce from Theorem \ref{cyclesinterval}
 that $C_{[m]}(\sigma)$ is \emph{uniformly}
close to $\log m$ for most $\sigma\in \cS_n$.

\begin{thm}\label{cyclesinterval_allm}
Let $2\le \xi \le n$.  With probability $1-O(1/(\log \xi)^{1/3})$, we have 
\[
|C_{[m]} - \log m| < 2\sqrt{\log m \log\log m} \quad (\xi \le m\le n).
\]
\end{thm}

Our proof is based on the analogous proof for the normal 
distribution of prime factors of integers given in
\cite[Ch. 1]{Divisors}.
When $m$ is bounded, $C_{[m]}(\sigma)$ has a discrete distribution
which is approximately Poisson with parameter $H_m$.
Slightly better bounds than those in Theorem
 \ref{cyclesinterval_allm} are attainable, based on ideas stemming from the Law of the Iterated Logarithm from probability
theory.  Essentially one can replace the factor  $\log\log m$ with $\log\log\log m$. See e.g., \cite{Manst-LIL} for a specific statement;
see also \cite[Theorem 11]{Divisors} for the analogous 
statement and proof for prime factors of integers.

Theorem \ref{cyclesinterval_allm} 
also tells us about the normal behavior of $D_j(\sigma)$, the length of the $j$-th smallest cycle of $\sigma$ (note that $D_j(\sigma)=D_{j+1}(\sigma)$ for some $j$ when $\sigma$ has cycles of the same length).
Since a typical permutation
$\sigma\in \cS_n$ has about $\log m$ cycles of length $\le m$,
we expect that $D_j(n) \approx \er^j$.

\begin{thm}\label{Djn}
Let $1\le \theta \le \log n$.  With probability $1-O(\theta^{-1/3})$, we have
\[
|\log D_j(\sigma) - j| < 3\sqrt{j\log j} \qquad (\theta \le j \le C(\sigma)).
\]
\end{thm}

We conclude this subsection with a sharp lower bound for $\PR(C(\sigma)=k)$.
This estimate is not new, but will be needed 
in section \ref{sec:conditioning}.

\begin{thm}\label{allsets-local}
We have
\be\label{lower-smallk}
\PR(C(\sigma)=k) \ge \frac{H_n^{k-1}}{n(k-1)!}\(1- \frac{k-1}{\log n}\) \qquad (1\le k < \log n).
\ee
For each fixed  $A>1$, there is a 
constant $c(A)>0$ such that for large enough $n$ (depending on $A$),
\[
\PR(C(\sigma)=k) \ge c(A) \frac{H_n^{k-1}}{(k-1)!}\er^{-H_n}
\qquad (1\le k\le A\log n).
\]
\end{thm}

In particular, taking Corollary \ref{cycles-single-set} and \eqref{lower-smallk} together
establishes the asymptotic
\be\label{allsets-asym}
\PR(C(\sigma)=k) \sim \frac{H_n^{k-1}}{n(k-1)!} \qquad (k=o(\log n), n\to\infty),
\ee
recovering a result of Moser and Wyman \cite{MoserWyman}
(the authors utilized generating functions and contour integration).  We note that \eqref{allsets-asym}
differs from the prediction of the Poisson model by 
a factor \[
\frac{1}{n}\er^{H_n} \sim \er^{\gamma} \qquad (n\to\infty).
\]

Theorem \ref{thm:Goncharov}, Theorem \ref{cycles_sets}, Theorem \ref{single-set-tails}, Theorem \ref{cyclesinterval}, Theorem \ref{cyclesinterval_allm},
Theorem \ref{Djn} and Theorem \ref{allsets-local} will be proved in section
\ref{sec:local}.

\subsection{Conditioning on the total number of cycles}

If we restrict attention to permutations with $k$ total cycles,
we may obtain analogous theorems
about the distribution of $C_I(\sigma)$.
We focus on the ``normal'' case when $k=O(\log n)$
and prove an  analog of 
Theorem \ref{single-set-tails}.
We expect that $C_I(\sigma)$ will have roughly a binomial
distribution with parameter $p=H(I)/H_n$,
since if $X,Y$ are independent Poisson random variables with parameters
$\lambda_1,\lambda_2$, respectively, then 
\begin{align*}
\PR ( X = \ell | X+Y = k) = 
\binom{k}{l} \pfrac{\lambda_1}{\lambda_1+\lambda_2}^\ell
\pfrac{\lambda_2}{\lambda_1+\lambda_2}^{k-\ell}.
\end{align*}

Without loss of generality, we may assume that $H(I)\le \frac12 H_n$, else replace $I$ by $[n] \setminus I$.

\begin{thm}\label{conditional-tails}
Fix $A>1$.
Let $I$ be a nonempty, proper subset of $[n]$ with $H(I)\le \frac12 H_n$, suppose $
2\le k\le A\log n$, and define let $p=H(I)/H_n \le \frac12$.
For any $0\le \psi \le \sqrt{p(1-p)(k-1)}$ we have
\[
\PR \Big( |C_I(\sigma)-p(k-1) | \ge \psi\sqrt{p(1-p)(k-1)}\; \Big| \; C(\sigma)=k \Big) =O_A \(\er^{-\frac13 \psi^2}\),
\]
the implied constant depending only on $A$.
\end{thm}

Theorem \ref{conditional-tails} will be proved in section \ref{sec:conditioning}.
 We also mention here work of Mez\H o and Wang \cite{MezoWang},
 who found an asymptotic for the number of permutations 
 with exactly $k$ cycles and all cycles having length $>m$,
 for fixed $k$ and $m$ with $n\to \infty$.

%
\subsection{Permutations without small cycles.}
%

Sharp bounds on $\PR(C_{[m]}(\sigma)=0)$ are a
key to establishing the Poisson model.
The model predicts that $\PR(C_{[m]}(\sigma)=0)$
should be about $\er^{-H_m}$,
and Corollary \ref{cycles-single-set} contains an upper bound
close to this.
 This cannot be expected to hold for large $m$, for example 
 $\PR(C_{[m]}(\sigma)=0)=1/n$ if $m \ge n/2$ since a permutation lacking cycles of length at most $m$ must be a single $n$-cycle.
  In fact, when $n/m$ is small, there is an asymptotic formula
 $\PR(C_{[m]}(\sigma)=0) \sim \omega(n/m)/m$ ($n\to\infty$, $m\to\infty$) where $\omega$ is Buchstab's function
 and $\omega(u)\to \er^{-\gamma}$ as $u\to\infty$ \cite[Theorem 5]{granville}.  This is analogous to the problem
 of counting integers $n\le x$ with no prime factor $\le x^{1/u}$
 (see \cite[Ch. III.6]{Tenbook}).

Our focus  is to prove that
$\PR(C_{[m]}(\sigma)=0)$ is very close to $\er^{-H_m}$ when 
$n/m$ is large.

 \begin{thm}\label{nosmallcycles2}
 Let $1\le m\le n$.  Then
 \[
 \PR(C_{[m]}(\sigma)=0) = \er^{-H_{m}} \(1 + O(\er^{-g(n/m)}) \),
 \]
 where $g(x)=0$ for $1\le x\le 20$ and for $x > 20$,
 \[
g(x) = x\log x - x\log\log\log x + O(x).
 \]
 \end{thm}
 
 Theorem \ref{nosmallcycles2} will be proved in section \ref{sec:nosmall}.
 
 Historically, the relation $\lim_{n\to \infty} \PR(C_{[m]}(\sigma)=0) \to \er^{-H_m}$, for $m$ fixed, is due to Gruder \cite{Gruder}.
 Exact asymptotics for $\PR(C_{[m]}(\sigma)=0)-\er^{-H_m}$
 have been obtained by Petuchovas
 \cite{Pet16, Pet18}, using generating functions \eqref{genfcn}
 and a lengthy argument based on contour integration.
 Our method is much simpler and is based on sieve methods in
 number theory.

 \subsection{Permutations without large cycles}
 
 The distribution of permutations without
 large cycles is very different from that predicted
 by the Poisson model.  If $\sigma$ has no large cycles,
 the fact that the
 cycle lengths must sum to $n$ implies that
 $\sigma$ must contain a very large number of smaller cycles, and this is a much rarer event.  We define
 \[
 \nu(n,m) = \PR(C_{\{m+1,\ldots,n\}}(\sigma)=0).
 \]
 
\begin{thm}\label{nolargecycles}
For $1\le m\le n$ we have
\[
\nu(n,m) \le \er^{-u\log u + u -1}, \qquad u=n/m.
\]
\end{thm}

 This bound is reasonably sharp throughout the
range $1\le m\le n$.  For example, when $m=1$, Stirling's formula implies
\[
\nu(1,m) =\frac{1}{n!}
\sim \frac{\er^{-n\log n + n}}{\sqrt{2\pi n}} \qquad
(n\to \infty).
\]
When $m=2$, Chowla, Herstein and Moore \cite{CHM} showed
an asymptotic for $\nu(2,m)$ which implies that
\[
\nu(2,m) = \er^{-(n/2)\log (n/2)+O(n)}.
\]
At the opposite extreme, when $n/m=u$ is bounded, then $\nu(n,m) \sim \rho(u)$ as $n\to\infty$ by Goncharov \cite{Gon44}, where $\rho$ is
 the Dickman function \cite{dickman}, the unique continuous solution of the 
 differential-delay equation
 \be\label{rho-recur}
\rho(u) = 1 \;\;\; (0\le u\le 1); \qquad u\rho'(u)=-\rho(u-1) \quad (u>1).
\ee
de Bruijn \cite{debruijn} found a precise asymptotic for $\rho(u)$ as $u\to \infty$.
In particular 
\[
\rho(u)=\er^{-u\log u-u\log\log (3u)+O(u)}.
\]
See also \cite{Tenbook}, Ch. III.5.4.

Using complex analytic methods starting from \eqref{genfcn}, Manstavi\v{c}ius and Petuchovas \cite{ManPet16} found
more precise asymptotics for $\nu(n,m)$ throughout the range $1\le m\le n$.
Their methods are motivated by the analogous problem of
counting integers lacking large prime factors,
see \cite[Ch. III.5]{Tenbook}.
Our next result, which has a very short proof, provides an asymptotic in  
large range of $n,m$.

\begin{thm}\label{nolargecycles-dickman}
For all $n\ge m\ge 1$ we have
\be\label{nu-bounds}
\rho\pfrac{n}{m} \le \nu(n,m) \le  \rho\pfrac{n+1}{m+1}.
\ee
\end{thm}

Theorems \ref{nolargecycles} and \ref{nolargecycles-dickman} will be proved in
section \ref{sec:nolarge}.

We have
\be\label{rho-local}
\rho(u-v) = \rho(u) \er^{O(v\log u)} \qquad (u\ge 2, 0\le v\le 1).
\ee
This follows from strong asymptotics for $\rho(u)$, e.g. \cite[Theorem III.5.13]{Tenbook}.  We give a short, direct deduction of \eqref{rho-local} in the Appendix.
Since $\frac{n}{m}-\frac{n+1}{m+1} \le \frac{n}{m^2}$ we deduce the following.

\begin{cor}\label{nu-asym}
We have
\[
\nu(n,m) \sim \rho(n/m)\qquad (m \le n = o(m^2/\log m), m\to \infty).
\]
\end{cor}

Corollary \ref{nu-asym} recovers Theorem 4 of \cite{ManPet16}.
 When $n \gg m^2/\log m$, $\nu(n,m) \not\sim \rho(n/m)$, 
the asymptotic having a different shape; see \cite[Theorem 2.4]{Pet16}.
Thus, the range of $n$ in Theorem \ref{nu-asym} is best possible.

When $n/2 \le m\le n$, $\sigma$ has at most
one cycle of length $k\in (m,n]$, thus
\be\label{nu-smalln}
\nu(n,m) = 1 - 
\sum_{m<k\le n} \E C_k(\sigma) = 1-(H_n-H_m).
\ee
In particular, when $m=50, n=100$, this helps to solve the ``100 prisoners problem'' \cite{100prisoners}: There are 100 prisoners, numbered to 100.
 The numbers from 1 to 100 are placed in 100 unmarked boxes.
Each prisoner is allowed to open 50 of the boxes, and no communication between prisoners is allowed.  If every prisoner
finds his own number then they all go free.  Although it appears hopeless, there is a strategy that will work about 31\% of the time.
If the boxes are labeled 1,...,100 on the outside, the mapping from external label to internal number is a permutation of $[100]$.  With probability $1-H_{100}+H_{50} \approx 0.31$, the permutation contains no cycles of length more than 50. In this case, if 
 every prisoner follows the cycle starting with his own number (first opens the box labeled on the outside with his number, then opens the box number that he finds in the first box, etc), he'll find his number
 inside one of the boxes after no more than 50 openings.

The limiting relation $\lim_{n\to \infty} \nu(n,\fl{n/u}) = \rho(u)$
was first proved by Knuth and Trabb Pardo \cite{KTP}, 46 years after
Dickman \cite{dickman} showed the analogous statement for prime factors.
The joint distribution of the lengths of the $r$ largest cycles of $\sigma$, with $r\ge 1$ fixed,
has also received considerable attention (see, e.g., \cite{ABT, LS, vershik}), but 
we will not discuss it here.  We also mention the survey paper \cite[Section 3.10,3.11]{lagarias}, which has more extensive historical
information about work on the distribution of the smallest and largest cycles.

\subsection{Poisson approximation of small cycle lengths}

Let $1\le k\le n$ and
consider the problem of modeling 
\[
\cC_k=(C_1(\sigma),\ldots,C_k(\sigma))\]
by the random vector 
\[
\cZ_k=(Z_1,\ldots,Z_k), 
\]
where $Z_1,\ldots,Z_k$ are independent Poisson random variables
with parameters $1,\frac12,\ldots,\frac{1}{k}$, respectively.
We especially desire a good approximation when $k$ is large, as opposed
to bounded (ref. Theorem \ref{cycles_sets}).
We express our results in terms of the Total Variational Distance
$d_{TV}(X,Y)$ between two random variables $X$ and $Y$ taking values
in a discrete space $\Omega$,
defined by 
\be\label{dTV}
d_{TV}(X,Y) := \sup_{U\subset \Omega}  \PR(X\in U) - \PR(Y\in U).
\ee

\begin{thm}\label{Poisson_smallcycles}
Let $1\le k\le n$.  Then $$d_{TV} (\cC_k, \cZ_k) \le \er^{-f(n/k)},$$ where $f(x)=0$ for $x\le 20$ and  for $x\ge 20$ we have
\[
f(x)=x\log x - x \log\log\log x  + O(x).
\]
\end{thm}

Theorem \ref{Poisson_smallcycles} will be proved in section \ref{sec:Kubilius}.

Theorem \ref{Poisson_smallcycles} is slightly weaker than
the main theorem of Arratia and Tavar\'e \cite{AT92},
which states that $d_{TV} (\cC_k, \cZ_k) \le \er^{-(n/k)\log(n/k)+O(n/k)}$.
Sharper bounds are known, and are expressed in terms of the 
Dickman and Buchstab functions (see \cite{ManPet16, Pet16}).
Our proof is significantly shorter than either of these treatments.

\medskip

We immediately obtain the following corollary, by grouping together integers into sets.

\begin{thm}\label{poisson_cycles}
Let $I_1,\ldots, I_m$ be disjoint subsets of $[k]$, with $k\le n$.  Then, for any 
set $\cJ \subseteq \NN_0^{m}$,
\[
\PR\Big( (C_{I_1}(\sigma),\ldots,C_{I_m}(\sigma)) \in \cJ \Big) = \PR\Big( (Y_1,\ldots,Y_m) \in \cJ\Big) + O(\er^{-f(n/k)}),
\]
where for each $i$,
$Y_i$ is Poisson with parameter $H(I_i)$, and $Y_1,\ldots,Y_m$
are independent.
\end{thm}

\subsection{Central Limit Theorems}
Combining Theorem \ref{poisson_cycles} with the Central
Limit Theorem for Poisson variables (Theorem \ref{Poisson_CLT} below) establishes a Central
Limit Theorem for the count of cycles whose lengths lie in an
arbitrary set $I\subset [n]$.

\begin{thm}\label{CLT_cycles}
Let $I\subset [n]$ with $H(I)\ge 3$.  Uniformly for all 
$I$ and any real $w$,
\[
\PR \( C_{I}(\sigma)\le H(I) + w \sqrt{H(I)} \) = \Phi(w) + 
O\pfrac{\log H(I)}{\sqrt{H(I)}},
\quad \;\; \Phi(w)=\frac{1}{\sqrt{2\pi}} \int_{-\infty}^w \er^{-\frac12 t^2}\, dt.
\]
\end{thm}

The special case $I=[n]$ was established by
Goncharov \cite{Gon44}, without a specific rate of convergence.
Goncharov analyzed carefully the asymptotics of the
Stirling number of the first kind, $s(n,m)$, the absolute value of which counts the
number of permutations $\sigma \in \cS_n$ with $C(\sigma)=m$.
Since $H_n=\log n + O(1)$ and $\Phi$ has bounded derivative,
we quickly arrive at the following.

\begin{thm}\label{CLT_all_cycles}
Let $n\ge 100$ and $w$ be real.  Then
\[
P\( C(\sigma) \le \log n + w\sqrt{\log n} \) = \Phi(w) + O\pfrac{\log\log n}{\sqrt{\log n}}.
\]
\end{thm}

The big-$O$ term in Theorem \ref{CLT_cycles}
 cannot be made smaller than $1/\sqrt{H(I)}$
since $C_I(\sigma)$ is integer valued, and
thus the left side is constant in intervals of $w$ of length
$1/\sqrt{H(I)}$, while $\Phi'(w) \gg 1$ if $w$ is bounded.
We remark that  when $H(I)$ is bounded, $C_I(\sigma)$ is expected to
have Poisson distribution with small parameter, and this cannot be approximated by a Gaussian.

We also derive that the $j$-th smallest cycle
of $\sigma$, denoted $D_j(\sigma)$ (with ties allowed), also obeys the
Gaussian law, refining Theorem \ref{Djn}.

\begin{thm}\label{jth-smallest-cycle}
Uniformly for $j$ in the range
\[
1\le j \le \log n - \sqrt{(\log n)\log\log n}
\]
and for any real $w$,
\[
\PR \Big(\log D_{j}(\sigma) \le j + w\sqrt{j}\Big) = \Phi(w) 
+O\pfrac{\log (2j)}{\sqrt{j}}.
\]
\end{thm}

The analogous statement for the $j$-th smallest prime factor of
an integer, without a rate of convergence, was proved by Galambos \cite{Galambos76}.

Theorems \ref{CLT_cycles} and \ref{jth-smallest-cycle}
will be proved in section \ref{sec:CLT}.

\subsection{Fixed sets and divisors of permutations}

A \emph{fixed set}  of a permutation $\sigma\in \cS_n$ is a subset of $[n]$ fixed by $\sigma$.
 A fixed set corresponds to a product of some subset of the cycles in $\sigma$
(we include both the empty set and the whole set $[n]$ as fixed sets).  These play the
same role for permutations as divisors do for integers.  The existence of fixed sets
of a particular size has applications to various questions in combinatorial group theory,
such as generation of $\cS_n$ by random permutations and the distribution of
transitive subgroups of $\cS_n$.
See e.g. \cite{camkan, dfg08, dixon92, EFG1, EFG2, EFK, FGK, LP, PPR, weingartner}.

We begin with a simple result about $2^{C(\sigma)}$, which counts the
number of fixed sets of $\sigma$, equivalently, the number of divisors of $\sigma$.

\begin{thm}\label{exp-numdiv}
 $\E 2^{C(\sigma)} = n+1$.
\end{thm}

By contrast, we know that $C(\sigma) \sim \log n$ for most $\sigma\in
\cS_n$ (for example, from Theorem \ref{CLT_all_cycles}), 
and therefore for most $\sigma\in \cS_n$,
$2^{C(\sigma)} \approx 2^{\log n} = n^{\log 2}$, much smaller than $n$.

\bigskip

A basic problem is to estimate $i(n,k)$, the probability that $\sigma\in \cS_n$ fixes some set of size $k$.   Equivalently, what is the probability that the cycle decomposition of $\sigma$ contains disjoint cycles with lengths summing to $k$?  Evidently, $i(n,k)=i(n,n-k)$, thus it
suffices to bound $i(n,k)$ for $k\le n/2$.
Sharpening earlier bounds due to Diaconis, Fulman and Guralnick \cite{dfg08}, {\L}uczak and Pyber \cite{LP} and by Pemantle, Peres, and Rivin \cite[Theorem 1.7]{PPR}, the author with Eberhard and Green
\cite{EFG1} proved that
\be\label{ink}
 \frac{1}{k^{\cE}(1+\log k)^{3/2}} \ll
i(n,k) \ll \frac{1}{k^{\cE}(1+\log k)^{3/2}},
\quad \cE = 1- \frac{1+\log\log 2}{\log 2}=0.08607\ldots,
\ee
uniformly for $1\le k\le n/2.$  A full asymptotic is not known.

This is the permutation analog of counting integers with a 
divisor in a given interval, see e.g. \cite{F08,F08b},
and is related to the Erd\H os multiplication table problem
(\cite{Erdos55,Erdos60}),
that of estimating the 
number, $A(N)$, of \emph{distinct} products of the form $ab$ with $a\le N$, $b\le N$. 
The full proof of \eqref{ink} is rather complicated.
However, using the tools we have developed in this paper, we can quickly obtain an upper bound which is close to optimal.

\begin{thm}\label{thm:fixed-set}
Uniformly for $1\le k\le n/2$ we have
\[
i(n,k) \ll \frac{1}{k^{\cE}}.
\]
\end{thm}

%
\section{Preliminaries}

The following standard bounds are stated without proof.

\begin{lem}\label{harmonic}
The harmonic sums $H_n$ satisfy

(i) $\log n \le H_n \le 1 +\log n$;

(ii) $H_n = \log n + \gamma + O(1/n)$,
where
$\gamma=0.57721566\ldots$ is Euler's constant.
\end{lem}

\begin{lem}[Stirling's formula]\label{Stirling}
We have $n! \ge (n/\er)^n$ and the asymptotic
\[
n! = \sqrt{2\pi n} (n/\er)^n (1+O(1/n)) \qquad (n\ge 1).
\]
\end{lem} 

\begin{lem}[Inclusion-exclusion]\label{incl-excl}
Let $a$ be a non-negative integer. For $0\le m\le k$,
\begin{align*}
\one(a=m) &= \sum_{r=m}^\infty (-1)^{r-m} \binom{r}{m} \binom{a}{r}
\\ &= \sum_{r=m}^k (-1)^{r-m} \binom{r}{m} \binom{a}{r} +
(-1)^{k+1-m} \binom{a}{m} \binom{a-m-1}{k-m},
\end{align*}
where the final term is
 at most $\binom{a}{k+1}\binom{k+1}{m}$ in
absolute value.
\end{lem}

The final claim comes from the inequality $\binom{a-m-1}{k-m} \le \binom{a-m}{k-m+1}$.

\begin{lem}[Poisson tails; see Norton {\cite[Section 4]{Norton76}}]\label{Poisson_tails}
Let $X$ be Poisson with parameter $\lambda$. Then
\begin{align*}
\PR(X\le \a \lam) &\le \min\bigg(1, \frac{1}{(1-\a)\sqrt{\a \lambda}}  \bigg)\er^{-Q(\a)\lam} \quad (0\le \a\le 1), \\
\PR(X\ge \a \lam) &\le \min\bigg( 1, \frac{1}{\a-1} \sqrt{\frac{\a}{2\pi \lam}}\bigg) \er^{-Q(\a)\lam} \quad (\a\ge 1),
\end{align*}
where $Q(x) = \int_1^x \log t\, dt = x\log x - x + 1.$
Furthermore,
\be\label{Qx_crude}
\frac{x^2}{3} \le Q(1+x) \le x^2 \quad (|x| \le 1)
\ee
and, when $0< x_1 \le x_2\le 1$ we have
\be\label{Qnear0}
Q(x_1)-Q(x_2) \le (-\log x_1)(x_2-x_1). 
\ee
\end{lem}


%
\section{Binomial moments}\label{sec:binom}
%
%

We begin by proving a special case of Theorem \ref{cycles_sets_expectation},
where each set $I_j$ is a singleton.  This is Theorem 7 in \cite{watterson}.

\begin{lem}\label{cycles}
Let $m_1,\ldots,m_n$ be non-negative integers with $m_1+2m_2+\cdots+nm_n\le n$.  Then
 \[
\E \prod_{j=1}^n \binom{C_j(\sigma)}{m_j} = \prod_{j=1}^n \frac{(1/j)^{m_j}}{m_j!}.
\]
If  $m_1+2m_2+\cdots+nm_n > n$, then the left side is zero.
\end{lem}
\begin{proof}
The second assertion is obvious, since the only way for the product on the left to be positive is for the sum of the cycle lengths to exceed $n$. Now assume that
$m_1+2m_2+\cdots+nm_n\le n$. 
 The number of ways of choosing from $[n]$ a disjoint collection 
  of $m_1$ $1-$element sets, $m_2$  $2-$element sets,
$\ldots$, $m_n$ $n-$element sets is equal to
\[
 \binom{n}{\underbrace{1 \cdots 1}_{m_1} \underbrace{2 \cdots 2}_{m_2} \cdots \underbrace{n\cdots n}_{m_n} \, t} 
 \frac{1}{m_1! \cdots m_n!} = \frac{n!/t!}{\prod_{j=1}^n (j!)^{m_j} m_j!},
 \]
where $t=n-(m_1+2m_2+\cdots+nm_n)$.
A $k$-element set may be arranged into a cycle in $(k-1)!$ ways.
Thus, 
the number of ways to arrange the elements of these sets into cycles is $(0!)^{m_1} (1!)^{m_2} \cdots (n-1)!^{m_k}$.
Finally, the $t$ elements not used in any of these cycles may be permuted in $t!$ ways.
\end{proof}

This special case suffices to prove Theorems \ref{thm:genfcn} and \ref{thm:Cauchy}.

\begin{proof}[Proof of Theorem \ref{thm:Cauchy} (Cauchy's Theorem)] 
Apply Lemma \ref{cycles}, noting that $\binom{C_j(\sigma)}{m_j}\ne 0$
for all $j$ if and only if $C_j(\sigma)=m_j$ for every $j$.
\end{proof}

\begin{proof}[Proof of Theorem \ref{thm:genfcn}]
Using Cauchy's formula, we have
\begin{align*}
\sum_{n,k} \PR_n \big(C_I(\sigma)=k,C_{[n]\setminus I}(\sigma)=0\big) x^n y^k &=
\sum_{n,k} x^n y^k \ssum{\sum_{i\in I} a_i = k \\ \sum_{i\in I} ia_i=n}
\prod_{i\in I} \frac{(1/i)^{a_i}}{a_i!} \\
&= \sum_{a_i\ge 0 : i\in I} \frac{x^{\sum ia_i} y^{\sum a_i} (1/i)^{a_i}}{\prod a_i!}\\
&=\exp \bigg\{ y \sum_{i\in I} \frac{x^i}{i} \bigg\}.\qedhere
\end{align*}
\end{proof}

\begin{proof}[Proof of Theorem \ref{cycles_sets_expectation}]
Consider a set $A$ of size $C_{I_j}(\sigma)$,
and partition $A$ into subsets $A_r$, where
$|A_r|=C_r(\sigma)$ for $r\in I_j$.  Then
\[
\prod_{j=1}^k\binom{C_{I_j}(\sigma)}{m_j} = \sum_{\eqref{cycles-sets-2}}\;
\prod_{j=1}^k \prod_{r\in I_j} \binom{C_r(\sigma)}{m_{j,r}},
\]
where the summation is over tuples $(m_{j,r})_{1\le j\le k,r\in I_j}$
satisfying the system
\be\label{cycles-sets-2}
\sum_{r\in I_j} m_{j,r} = m_j \quad (1\le j\le k). 
\ee
Thus,
\be\label{cycles-sets-1}
\E \prod_{j=1}^k \binom{C_{I_j}(\sigma)}{m_j} = \sum_{\eqref{cycles-sets-2}} \E \prod_{j=1}^k \prod_{r\in I_j} \binom{C_r(\sigma)}{m_{j,r}}.
\ee
Using Lemma \ref{cycles}, the expectation on the right side of \eqref{cycles-sets-1} equals
\[
\prod_{j=1}^k \prod_{r\in I_j} \frac{(1/r)^{m_{r,j}}}{m_{r,j}!}
\]
provided that
\be\label{cycles-sets-3}
\sum_{j=1}^k \sum_{r\in I_j} r m_{r,j} \le n, 
\ee
and is zero otherwise.  

If $\sum_{j=1}^k m_j \max(I_j)\le n$, then \eqref{cycles-sets-3} will always be satisfied
as long as \eqref{cycles-sets-2} holds,
and therefore
\[
\E \prod_{j=1}^k \binom{C_{I_j}(\sigma)}{m_j} =
 \prod_{j=1}^k \sum_{\eqref{cycles-sets-2}} \prod_{r\in I_j} \frac{(1/r)^{m_{r,j}}}{m_{r,j}!} = \prod_{j=1}^k \frac{H(I_j)^{m_j}}{m_j!},
\]
as claimed.  On the other hand, if
 $\sum_{j=1}^k m_j \max(I_j) > n$, then there is some choice of the parameters
$(m_{j,r})$ satisfying \eqref{cycles-sets-2} but violating \eqref{cycles-sets-3}, 
and the left side is strictly less than the right side.
Specifically, we may take $m_{j,\max I_j}=m_j$ for each $j$
and $m_{j,r}=0$ otherwise. 
\end{proof}

%
\section{Local limit theorems}\label{sec:local}
%
%

\begin{proof}[Proof of Goncharov's local limit theorem, Theorem \ref{thm:Goncharov}]
By Lemma \ref{incl-excl} and Lemma \ref{cycles}, we obtain
\begin{align*}
\PR \big( C_j(\sigma)=m \big) &= \E
 \sum_{r=m}^\infty (-1)^{r-m} \binom{r}{m} \binom{C_j(\sigma)}{r} 
=  (-1)^m \sum_{r=m}^{\fl{n/j}} \binom{r}{m}  \frac{(-1/j)^r}{r!}.  
\end{align*}
The desired equality follows by setting $r=h+m$.
\end{proof}

While Theorem \ref{thm:Goncharov} provides a exact formula
for the local statistic $\PR(C_j(\sigma)=m)$, an analogous formula for
$\PR(C_I(\sigma)=m)$ with an arbitrary set $I$ will necessarily be far more complicated.  However, borrowing ideas from the theory
of averages of multiplicative functions in number theory, we 
give a relatively sharp upper bound for this quantity, and
more generally for the joint probability of 
$C_{I_j}(\sigma)=m_j$ for $j=1,\ldots,k$.

We begin with a rather complicated identity for the joint distribution
of the quantities $C_{I_i}$.

\begin{lem}\label{lem:CIjmj}
Let $I_1,\ldots,I_r$ be disjoint subsets of $[n]$ and $m_1,\ldots,m_r$
be non-negative integers.  Denote $I_0 = [n] \setminus (I_1\cup \cdots \cup I_r)$.
Let $\cT$ be the set of indices $i$ with $m_i>0$, together
with the number 0 if $I_0$ is nonempty.  Then
\[
  \PR\big( C_{I_j}(\sigma)=m_j\; (1\le j\le r) \big)
  = \frac{1}{n}\sum_{t\in \cT} \sum_{h \in I_t} 
  \ssum{b_1,\dots, b_n \geq 0 \\ b_1 + 2b_2 + \cdots + nb_n = n - h \\ \sum_{i\in I_j} b_i =m_j-\one(t=j),\  (1\le j\le r) } 
  \prod_{i=1}^n \frac{(1/i)^{b_i}}{b_i!}.
\]
\end{lem}

\begin{proof}
 Evidently
\[
  n \# \{\sigma \in \cS_n:  C_{I_1}(\sigma)=m_1,\ldots,C_{I_r}(\sigma)=m_r\} = \ssum{\sigma \in \cS_n \\ C_{I_j}(\sigma)=m_j \; (1\le j\le r) } \;\; \ssum{\alpha|\sigma \\ \a \text{ a cycle}} |\a|.
\]
Write $\sigma=\a \beta$ and let $h=|\a|$.   
 Thus, for some $t\in \cT$,  we have $|\a|=h\in I_t$
 and
\[
(C_{I_1}(\beta),\ldots,C_{I_r}(\beta))=(m_1-\one(t=1),\ldots,m_r-\one(t=r)).
\]
It is permissible to think of $\beta\in \cS_{n-h}$ and thus
\begin{align*}
  n  \# \{\sigma \in \cS_n: C_{I_1}(\sigma)=m_1,\ldots,C_{I_r}(\sigma)=m_r\}
  &= \sum_{t\in \cT} \sum_{h\in I_t} \; \ssum{\a\in\cS_n, |\a|=h \\ \a\text{ a cycle}} h \ssum{\beta\in \cS_{n-h}\\ C_{I_i}(\beta)=m_i-\one(t=i), (1\le i\le r)} 1  \\
  &=  \sum_{t\in \cT} \sum_{h\in I_t} \frac{n!}{(n-h)!} \ssum{\beta\in \cS_{n-h} \\ C_{I_i}(\beta)=m_i-\one(t=i), (1\le i\le r) } 1.
\end{align*}
Now subdivide the sum according the cycle type $(b_1,\ldots,b_{n})$ of
the permutation $\beta$, use Cauchy's formula (Thm. \ref{thm:Cauchy}) to 
count such permutations for each type, and divide by $n$.  The desired identity follows.
\end{proof}

\begin{proof}[Proof of Theorem \ref{cycles_sets}]
The right side in Lemma \ref{lem:CIjmj} is at most
 \[
 \frac{1}{n} \sum_{t\in \cT}
  \ssum{b_1,\dots, b_n \geq 0 \\ \sum_{i\in I_j} b_i =m_j-\one(t=j) \; (1\le j\le r)  } \prod_i \frac{(1/i)^{b_i}}{b_i!} := \frac{Y}{n},
\]
  say.   By the multinomial theorem,
  \begin{align*}
 Y & =    \sum_{t\in \cT} \ssum{b_i\ge 0 \; (i\in I_1\cup \cdots \cup I_r) \\ \sum_{i\in I_j} b_i = m_j-\one(t=j) \; (1\le j\le r) } 
  \frac{1}{\prod_{i\in I_1\cup \cdots \cup I_r} b_i! i^{b_i}} \sum_{b_i\ge 0 \; (i\in I_0)} \frac1{\prod_{i\in I_0} b_i! i^{b_i}}\\
  &=  \sum_{t\in \cT} \frac{m_t}{H(I_t)}\prod_{j=1}^r \frac{H(I_j)^{m_j}}{m_j!} \er^{H(I_0)}.
\end{align*}
 The claimed bound now follows from
$H(I_0) = H_n - H(I_1)-\cdots-H(I_r)$.
\end{proof}

Later, we will sharpen the conclusion when $r=1$, $I_1=[k]$, $m_1=0$
(permutations lacking small cycles)
and when $r=1$, $I_1=\{k+1,\ldots,n\}$ and $m_1=0$ (permutations lacking
large cycles).

\begin{proof}[Proof of Theorem \ref{single-set-tails}]
For brevity, let $H=H(I)$.
For the first inequality, apply Corollary \ref{cycles-single-set}
for all $m\le \lambda H$, using $H_n\le \log n + 1$, followed by 
an application of Lemma \ref{Poisson_tails}.  This gives
\[
\PR( C_I(\sigma) \le \lambda H ) \le 2 \sum_{m\le \lambda H}
\er^{1-H} \frac{H^m}{m!} \le 2\er^{1-Q(\lambda)H}.
\]
The second inequality is similar.  We have
\begin{align*}
\PR( C_I(\sigma) \ge \lambda H+1 ) &\le  \sum_{m\ge \lambda H+1} \er^{1-H} \bigg( \frac{H^m}{m!} + \frac{H^{m-1}}{(m-1)!} \bigg) \\
&\le  2 \sum_{m\ge \lambda H} \er^{1-H} \frac{H^m}{m!} 
\le 2\er^{1-Q(\lambda)H}.
\end{align*}

The third assertion is trivial if $\psi \le 1$, thus we may assume that
$\psi>1$, and in particular that $H>1$.  
Define $\lambda^{\pm}$ by
\[
\lambda^- H = H-\psi\sqrt{H}, \qquad
\lambda^+ H + 1 = H + \psi\sqrt{H}.
\]
In particular, $0 \le \lambda^- \le 1 \le \lambda^+ \le 2$.
Apply the first inequality in Theorem \ref{single-set-tails} with $\lambda=\lambda^-$ and the second inequality in  Theorem \ref{single-set-tails} with
$\lambda=\lambda^+$, obtaining
\[
 \PR \( |C_{I}(\sigma) - H| \ge  \psi \sqrt{H} \)
\le 2 \er^{1-Q(\lambda^-)H} + 2\er^{1-Q(\lambda^+)H}. 
\] 
By \eqref{Qx_crude},
\[
Q(\lambda^-) = Q\(1-\frac{\psi}{H^{1/2}}\) \ge \frac{\psi^2}{3 H}
\]
and
\[
Q(\lambda^+) = Q\(1-\frac{\psi}{H^{1/2}} + \frac{1}{H} \) \ge 
\frac{1}{3H} (\psi - 1/\sqrt{H})^2 \ge \frac{\psi^2-2}{3H}
\]
and the third assertion follows, since $2\er + 2\er^{5/3} \le 20$.
\end{proof}

\begin{proof}[Proof of Theorem \ref{cyclesinterval}]
Let $I=[a,b] \cap \NN$, $H=H(I)$ and let $K$ be a sufficiently large constant.
The conclusions are trivial when $b/a \le K$, henceforth we assume that 
$b/a > K$.  By \eqref{Qx_crude}, the assertions are also trivial when
\[
1 - \frac{1}{\sqrt{\log(b/a)}} \le \lambda \le 1 +  \frac{1}{\sqrt{\log(b/a)}},
\]
and henceforth we assume that
\be\label{cor18-lam1}
|\lambda -1| >  \frac{1}{\sqrt{\log(b/a)}}.
\ee
 By Lemma \ref{harmonic},
\be\label{thm18-H}
H = \log(b/a) + O(1).
\ee
As the first assertion follows from Theorem \ref{single-set-tails}
if $\lambda=0$, we may assume that $\lambda>0$.

Firstly, suppose that $0< \lambda\le 1$ and that \eqref{cor18-lam1} holds.
If we define $\lambda'$ by
\[
\lambda \log(b/a) = \lambda'H,
\]
then $\lambda'\le 1$,
and thus by Theorem \ref{single-set-tails},
\[
\PR \( C_{I}(\sigma) \le \lam \log (b/a) \) =
\PR(C_I(\sigma) \le \lambda' H)\le 2\er^{1-Q(\lambda')H}.
\]
By \eqref{thm18-H},
\[
|\lambda-\lambda'| \ll \frac{\min(\lambda,\lambda')}{\log(b/a)}
\]
and hence \eqref{Qnear0} implies that
\[
Q(\lambda)-Q(\lambda') \ll (-\log \min(\lambda,\lambda') )  \frac{\min(\lambda,\lambda')}{\log(b/a)} \ll \frac{1}{\log (b/a)} \ll \frac{1}{H}
\]
and the first assertion follows.

The proof of the second bound is similar.  Suppose that $1\le \lambda \le \lambda_0$
and \eqref{cor18-lam1} holds.
If we define $\lambda'$ by 
\[
\lambda \log(b/a) = \lambda'H+1,
\]
then  $1\le \lambda' \le 2\lambda_0$ if $K$ is large enough.
Theorem \ref{single-set-tails} then implies that
\[
\PR \( C_{I}(\sigma) \ge \lam \log (b/a) \) =
\PR(C_I(\sigma) \ge \lambda' H+1)\le 2\er^{1-Q(\lambda')H}.
\]
By \eqref{thm18-H}, $|\lambda-\lambda'| \ll_{\lambda_0} \frac{1}{\log(b/a)} \ll 1/H$.
Since $Q'(x) \le \log(2\lambda_0)$ for $1\le x\le 2\lambda_0$,
we have 
\[
|Q(\lambda)-Q(\lambda')|  \ll_{\lambda_0} 1/H
\]
and the second assertion follows.

The final estimate follows from the first two, with $\lambda_0=2$,
and the bound \eqref{Qx_crude} for $Q(u)$.
\end{proof}

\begin{proof}[Proof of Theorem \ref{cyclesinterval_allm}]
We may assume that $\psi$ is sufficiently large.
Let
\[
k_1 = \fl{\log \xi}+1, \qquad k_2 = \fl{\log n},
\]
and for $k_1 \le k\le k_2$, let $t_k = \er^k$.  Put $t_{k_1-1}=\xi$
and $t_{k_2+1}=n$.  For brevity, write $C(\sigma;t) := \sum_{j\le t} C_j(\sigma)$.
 For each $k$, $k_1-1 \le k\le k_2+1$, let
$N_k(x)$ be the probability 
that
\be\label{omegantk}
|C(\sigma;t_k)-\log t_k| \ge 2\sqrt{(k-1)\log(k-1)}-1.
\ee
As $\log t_k = k+O(1)$ for all $t_k$ (including the endpoints),
\[
2\sqrt{(k-1)\log(k-1)}-1 = \psi \sqrt{\log t_k}, \quad \psi=2\sqrt{\log k}+O(1/\sqrt{k}). 
\]
Since $k$ is sufficiently large, for all $k\ge k_1$ we have
$\psi \le \sqrt{\log t_k}$.
By the third part of Theorem \ref{cyclesinterval},
\[
N_k(x) \ll \er^{-\frac13 \psi^2} \ll \frac{1}{k^{4/3}}.
\]
Summing over $k$, we see that the probability
that \eqref{omegantk} holds for some $k$ is bounded by
 $O(1/(\log \xi)^{1/3})$.   Now
suppose that \eqref{omegantk} fails for every $k$ with $k_1-1\le k\le k_2+1$.
Let $\xi \le t\le x$ and suppose that $t_k \le  t\le t_{k+1}$.  Evidently,
\[
C(\sigma;t_k) \le C(\sigma;t) \le C(\sigma;t_{k+1}).
\]
Since $\log t_k \ge k$ and $\log t_{k+1} \le k+1$, $k\le \log t \le k+1$.
By the failure of \eqref{omegantk} at every $k$,
\[
 C(\sigma;t)\ge \log t_k - 2\sqrt{(k-1)\log(k-1)}+1 \ge \log t-2\sqrt{\log t\log\log t}
\]
and
\[
C(\sigma;t) \le \log t_{k+1} + 2\sqrt{k\log k} - 1 \le  \log t+2\sqrt{\log t\log\log t}.\qedhere
\]
\end{proof}

\begin{proof}[Proof of Theorem \ref{Djn}]
We may suppose that $\theta\ge \theta_0$, where $\theta_0$ is a 
sufficiently large, absolute constant, for otherwise the conclusion
of the Corollary is trivial if the implied constant is large enough.  
Let $\xi = \fl{\er^{(2/3)\theta}}$.  By  Theorem
\ref{cyclesinterval_allm}, with probability $1-O(1/\theta^{1/3})$,
we have
\be\label{Djn:Cm}
|C_{[m]}(\sigma)-\log m| < 2\sqrt{\log m\log\log m} \qquad (\xi \le m\le n).
\ee
Also, by Corollary \ref{two-equal-large-cycles}, with
probability $1-O(1/\xi)$ all the cycles of $\sigma$ of length $\ge \xi$ 
have distinct lengths.
Now suppose that $\sigma$ is a permutation satisfying \eqref{Djn:Cm},
and such that the cycles of $\sigma$ with lengths $\ge \xi$ have distinct lengths.
We suppose that $\theta_0$ is so large that the right side of the inequality
in \eqref{Djn:Cm} is at most $\frac12 \log m$ for every $m\ge \xi$.  
In particular, 
\[
 C_{[\xi]}(\sigma) < \frac32 \log \xi \le \theta,
\]
that is, $D_\theta(\sigma) > \xi$.
Thus, we may apply \eqref{Djn:Cm} with $m=D_j(\sigma)$ for all $\theta\le j\le C(\sigma)$.  As the cycle lengths $\ge \xi$ are distinct,
we have $j=C_{[m]}(\sigma) > \frac12 \log D_j(\sigma)$ and hence
\[
|j - \log D_j(\sigma)| < 2\sqrt{\log D_j(\sigma)\log\log D_j(\sigma)} <
2\sqrt{2j\log(2j)} < 3\sqrt{j\log j}
\]
provided that $\theta_0$ is large enough (and hence $j$ is large enough).
\end{proof}

\begin{proof}[Proof of Theorem \ref{allsets-local}]
If $k=1$, $\PR(C(\sigma)=1)=1/n$.  Now suppose $k\ge 2$.
We begin with Lemma \ref{lem:CIjmj},
which implies that 
\be\label{allsets-equality}
n \cdot \PR(C(\sigma)=k)= \ssum{b_1,\ldots,b_n\ge 0 \\ 
b_1+2b_2+\cdots \le n \\ b_1+\cdots+b_n=k-1} 
\frac{1}{\prod_{i\le n} b_i! i^{b_i}}.
\ee
We restrict the
summations to $b_i=0\;\; (i>m)$ for
some parameter $m\in [1,n]$ to be chosen later.
Using
\[
\one(b_1+2b_2+\cdots + mb_m \le n) \ge \frac{n-(b_1+2b_2+\cdots + mb_m)}{n}
\]
and the multinomial theorem,
\begin{align*}
n\cdot \PR(C(\sigma)=k) &\ge \frac{1}{n} \ssum{b_1,\ldots,b_m\ge 0 \\ b_1+\cdots+b_m=k-1} \frac{n-(b_1+2b_2+\cdots+mb_m)}{\prod_{i\le m} b_i! i^{b_i}}\\
&=\frac{H_m^{k-1}}{(k-1)!} - \frac{m}{n}
\cdot \frac{H_m^{k-2}}{(k-2)!}\\
&= \frac{H_m^{k-1}}{(k-1)!}\bigg( 1 - \frac{m(k-1)}{nH_m} \bigg).
\end{align*}
When $1\le k \le \log n$, we take $m=n$ and note that
$H_m=H_n\ge \log n$.  This proves \eqref{lower-smallk}.

To obtain the 2nd part of Theorem \ref{allsets-local}, we 
fix $A \ge 2$ and take $m=n/(2A)$.  
We have $H_m = H_n + O(\log A)$  and $k\le A\log n \le A H_n$.
Hence, for $n$ large enough,
\[
\PR(C(\sigma)=k) \ge \frac{H_m^{k-1}}{3n(k-1)!} 
\ge c(A) \frac{H_n^{k-1}}{(k-1)!} \er^{-H_n} 
\]
for some positive $c(A)$.
\end{proof}



\section{Conditioning on the total number of cycles}\label{sec:conditioning}

We will use an explicit Chernoff bound for tails of the binomial distribution.
Denote by $\Bin(k,p)$ a binomial random variable corresponding to $k$
trials, and parameter $p\in [0,1]$.

\begin{lem}[{\cite[Lemma 4.7.2]{Ash}}]\label{binomial_tails}
If $0<p<1$ and $\beta\le p$ then we have
\[
\PR(\Bin(n,p)\le \beta n) \le \exp \left\{ - n \( \b \log \frac{\b}{p} + (1-\b) \log \frac{1-\b}{1-p} \) \right\}
\le \exp \bigg\{-\frac{(p-\beta)^2n}{3p(1-p)} \bigg\}.
\]
Replacing $p$ with $1-p$ we also have for $\beta\ge p$,
\[
\PR(\Bin(n,p)\ge \beta n) \le  \exp \bigg\{-\frac{(p-\beta)^2n}{3p(1-p)} \bigg\}.
\]
\end{lem}

\begin{proof}[Proof of Theorem \ref{conditional-tails}]
Apply
Theorem \ref{cycles_sets} with two sets: $I$ and $[n]\setminus I$.
Here $\eps=0$.  Divide the right side in Theorem \ref{cycles_sets} by $\PR(C(\sigma)=k)$, where a lower bound
is given in Theorem \ref{allsets-local}.  Set $p=H(I)/H_n$. 
Then, for $0\le h\le k$,  
\begin{align*}
\PR\Big(C_I(\sigma)=h | C(\sigma)=k \Big) &= \frac{\PR(C_I(\sigma)=h \, \land \, 
C_{[n]\setminus I}(\sigma)=k-h)}{\PR(C(\sigma)=k)}\\
&\ll_A \PR\big( \Bin(k-1,p)=h-1\big)+ \PR\big( \Bin(k-1,p)=h \big),
\end{align*}
Set $\beta^- = p - \psi\sqrt{p(1-p)/(k-1)}$.
By Lemma \ref{binomial_tails},
\[
\PR\Big(C_I(\sigma)\le \beta^{-}(k-1)\;\; \big|\;\; C(\sigma)=k \big) 
\ll_A \PR \Big( \Bin(k-1,p) \le \beta^-(k-1) \Big) \ll_A \er^{-\frac13 \psi^2}.
\]
Let $\beta^+ = p + \psi\sqrt{p(1-p)/(k-1)} - \frac{1}{k-1}$.
Since $0 \le \psi \le \sqrt{p(1-p)(k-1)}$,
\begin{align*}
\PR\Big(C_I(\sigma)\ge \beta^+(k-1)+1\;\; \big|\;\; C(\sigma)=k \Big) &\ll_A
\PR \Big( \Bin(k-1,p) \ge \beta^+ (k-1) \Big)\\
&\ll_A \exp\bigg\{  - \frac{(\psi\sqrt{p(1-p)}-1/\sqrt{k-1})^2}{3p(1-p)} \bigg\}\\
&\ll_A \er^{-\frac13 \psi^2}.
\end{align*}
This completes the proof.
\end{proof}

%
\section{Permutations without small cycles}\label{sec:nosmall}
%

\begin{proof}[Proof of Theorem \ref{nosmallcycles2}]
Our proof is based on the Brun-Hooley sieve \cite{BrunHooley}
from number theory.
Let $K \ge \er^{10}$ be fixed and sufficiently large and let $u=n/m$.  If $u \le K$, then
Corollary \ref{cycles-single-set} implies that $\PR(C_{[m]}(\sigma)=0)\ll 1/m$ and 
the conclusion follows.  Now assume that $u > K$ and let 
\[
D = \log u,
\]
so that $D\ge 10$.
Partition $[m]$ into intervals $I_j=[z_j,z_{j-1})\cap \NN$, where $z_j=m/D^j$,
$0\le j\le \cl{\frac{\log m}{\log D}}=t$. 
Let $k_1,\ldots,k_t$ be positive, even integers, subject to
\be\label{kj-cond}
k_j \ge 10\log D \;\; (1\le j\le t), \qquad \sum_{j=1}^t \frac{(k_j+1) m}{D^{j-1}} \le n.
\ee
With $\sigma$ fixed, let
\[
x_j = \one(C_{I_j}(\sigma)=0), \quad y_j = \sum_{r=0}^{k_j} (-1)^r \binom{C_{I_j}(\sigma)}{r}.
\]
By Lemma  \ref{incl-excl}, we have
\[
0\le y_j - x_j = \binom{C_{I_j}(\sigma)-1}{k_j} \le \binom{C_{I_j}(\sigma)}{k_j+1}.
\]
Using the elementary inequality
 \[
x_1\cdots x_t \ge y_1\cdots y_t - \sum_{\ell=1}^t (y_\ell-x_\ell) \sprod{j=1 \\ j\ne \ell}^t y_j,  
 \]
 together with $\PR(C_{[m]}(\sigma)=0) =  \E x_1\cdots x_t$, we thus obtain
\be\label{small-ME}
M - E \le \PR(C_{[m]}(\sigma)=0) \le M,
\ee
where
\[
M=\E y_1\cdots y_t, \qquad E = \E \sum_{\ell=1}^t \binom{C_{I_\ell}(\sigma)}{k_\ell+1} \prod_{j\ne \ell} y_j.
\]
The condition \eqref{kj-cond} implies that 
\be\label{kjmaxIj-2}
\sum_j (k_j+1) \max I_j \le n.
\ee
  Thus, by
Theorem \ref{cycles_sets_expectation},
\begin{align*}
M &= \ssum{r_1,\ldots,r_t \\ 0\le r_j\le k_j (1\le j\le t)} 
(-1)^{r_1+\cdots+r_t} \E \binom{C_{I_1}(\sigma)}{r_1} \cdots \binom{C_{I_t}(\sigma)}{r_t} \\
&=  \ssum{r_1,\ldots,r_t \\ 0\le r_j\le k_j (1\le j\le t)} 
(-1)^{r_1+\cdots+r_t} \prod_{j=1}^t \frac{H(I_j)}{r_j!} 
= \prod_{j=1}^t \Bigg( \sum_{r_j=0}^{k_j} \frac{(-H(I_j))^{r_j}}{r_j!}    \Bigg).
\end{align*}
Since $H(I_j)=\log D+O(1)$ for every $j$, 
and recalling \eqref{kj-cond}, we have 
\begin{align}
\sum_{r_j=0}^{k_j} \frac{(-H(I_j))^{r_j}}{r_j!} &= \er^{-H(I_j)} +
 O\pfrac{H(I_j)^{k_j+1}}{(k_j+1)!} \notag \\
&=\er^{-H(I_j)} \bigg( 1 + O \pfrac{D(\log D+O(1))^{k_j+1}}{(k_j+1)!}  \bigg) \notag
\\ &=\er^{-H(I_j)} \exp \Bigg[ O \pfrac{D(\log D+O(1))^{k_j+1}}{(k_j+1)!}  \Bigg].
\label{nosmall-yj}
\end{align}
Hence, the main term satisfies
\be\label{small-M}
M = \er^{-H_m} \exp\left[ O\(\sum_{j=1}^t \frac{D(\log D+O(1))^{k_j+1}}{(k_j+1)!}\) \right].
\ee
Similarly, using \eqref{kjmaxIj-2},  the error term satisfies
\begin{align*}
E &=  \sum_{\ell=1}^t \ssum{r_j (j\ne \ell) \\ 0\le r_j\le k_j (j\ne \ell))} 
(-1)^{\sum_{j\ne \ell} r_j} \E \binom{C_{I_\ell}(\sigma)}{k_\ell+1} 
\prod_{j\ne \ell} \binom{C_{I_j}(\sigma)}{r_j} \\
&=  \sum_{\ell=1}^t
 \frac{H(I_\ell)^{k_\ell+1}}{(k_\ell+1)!} \prod_{j\ne \ell}
\Bigg( \sum_{r_j=0}^{k_j} \frac{(-H(I_j))^{r_j}}{r_j!}  \Bigg) .
\end{align*}
Hence, by \eqref{nosmall-yj},
\be\label{small-E}
E=\sum_{\ell=1}^t \frac{\er^{H(I_\ell)}(\log D+O(1))^{k_\ell+1}}
{(k_\ell+1)!} \er^{-H_m} \exp\left[ O\(\sum_{j=1}^t \frac{D(\log D+O(1))^{k_j+1}}{(k_j+1)!}\) \right].
\ee
We now take
\[
k_j=k_1+2(j-1) \quad (j\ge 1), \;\; k_1=2\fl{\frac{D-1}{D} \cdot \frac{u}{2}}-6,
\]
and readily verify that the conditions \eqref{kj-cond} hold if
$K$ is large enough.
Thus, by Stirling's formula,
\begin{align*}
\sum_{j=1}^t \frac{D(\log D+O(1))^{k_j+1}}{(k_j+1)!} 
&\ll \frac{D(\log D+O(1))^{k_1+1}}{(k_1+1)!} \\
&\le \er^{-u\log u + u\log\log\log u + O(u)}
\end{align*}
and likewise
\[
\sum_{\ell=1}^t \frac{\er^{H(I_\ell)}(\log D+O(1))^{k_\ell+1}}
{(k_\ell+1)!} \le  \er^{-u\log u + u\log\log\log u + O(u)}.
\]
Inserting these last two bounds into 
\eqref{small-M} and \eqref{small-E}, and recalling \eqref{small-ME},
the proof is complete.
\end{proof}


\section{Permutations without large cycles}\label{sec:nolarge}

The traditional approach to the problem of
estimating the probability that a random permutation
has no cycle of size $>m$ is via
generating functions, e.g. Theorem \ref{genfcn}.
The sharpest results depend on a lengthy complex-analytic
argument, see \cite{ManPet16, Pet16}.

\begin{proof}[Proof of Theorem \ref{nolargecycles}]
Let  $w\ge 1$.  If $\sigma$ has no cycles of length $>m$, then
$\sum_{j=1}^m j C_j(\sigma)=n$ and hence
\[
\nu(n,m) \le \E 
w^{C_1(\sigma)+2C_2(\sigma)+\cdots+mC_m(\sigma)-n}.
\]
For $1\le j\le m$, write $w^j = 1 + (w^j-1)$.  By the binomial theorem
and Lemma \ref{cycles},
\begin{align*}
\nu(n,m) &\le w^{-n} \E 
\prod_{j=1}^m \Bigg( \sum_{k_j=0}^\infty (w^j-1)^{k_j} \binom{C_j(\sigma)}{k_j} \Bigg)\\
&= w^{-n} \sum_{k_1,\ldots,k_m\ge 0} (w-1)^{k_1}\cdots (w^m-1)^{k_m}
\E \binom{C_1(\sigma)}{k_1}\cdots \binom{C_m(\sigma)}{k_m}\\
&\le  w^{-n} \sum_{k_1,\ldots,k_m\ge 0} (w-1)^{k_1}\cdots (w^m-1)^{k_m} \prod_{j=1}^m
\frac{(1/j)^{k_j}}{k_j!}\\
&= w^{-n} \exp  \Bigg\{ \frac{w-1}{1} + \frac{w^2-1}{2}+\cdots
\frac{w^m-1}{m} \Bigg\}.
\end{align*}
A good all-purpose choice is $w=u^{1/m}$, where $u=n/m$.
The mean value theorem implies that 
\[
w^j=u^{j/m} \le 1 + (u-1)j/m \qquad  (1\le j\le m)
\]
 and hence
\be\label{powers-w}
 w-1 + \frac{w^2-1}{2} + \cdots + \frac{w^m-1}{m} \le \sum_{j=1}^m \frac{(u-1)j/m}{j}=u-1.
\ee
We conclude that
\[
\nu(n,m) \le u^{-n/m} \er^{u-1} =
\er^{-u\log u + u - 1}. \qedhere
\]
\end{proof}

For the proof of Theorem \ref{nolargecycles-dickman}, we need only very basic facts about
the Dickman function $\rho(u)$, namely that it is positive and decreasing.
These facts follow quickly from the definition plus the relation
\be\label{rho-int}
v \rho(v) = \int_{v-1}^v \rho(u)\, du \qquad (v\ge 1)
\ee
obtained by integrating \eqref{rho-recur} from $u=1$ to $u=v$.

\begin{proof}[Proof of Theorem \ref{nolargecycles-dickman}]
When $m\le n\le 2m$, the desired bounds \eqref{nu-bounds} follow from
\eqref{nu-smalln}, the fact that $\rho(u)=1-\log u$ for $1\le u\le 2$ and the easy inequalities
\[
 \log\pfrac{n+1}{m+1} = \int_{m+1}^{n+1}\frac{dt}{t} \le 
H_n - H_m \le \int_{n}^m \frac{dt}{t} = \log\pfrac{n}{m}.
\]
For larger $n$, we fix $m$ and argue by induction.  
For $1\le \ell\le m$, there are $\binom{n}{\ell}(\ell-1)!$ ways to form an 
$\ell-$cycle from $[n]$.  Hence
\begin{align*}
\nu(n,m) = \frac{1}{n!} \!\!\ssum{\sigma \in \cS_n \\ C_{(m,n]}(\sigma)=0}\frac{1}{n} \ssum{\tau|\sigma \\ \tau\text{ a cycle}} |\tau| &= \frac{1}{n\cdot n!} \sum_{\ell=1}^m \ell \binom{n}{\ell}(\ell-
1)! (n-\ell)! \nu(n-\ell,m)\\ &= \frac{1}{n} \sum_{k=n-m}^{n-1} \nu(k,m).
\end{align*}
Now fix $m\ge 1$, let $N\ge 2m+1$ and assume that \eqref{nu-bounds}
holds when $m\le n\le N-1$. 
Using \eqref{rho-int} and the monotonicity of $\rho$,
\begin{align*}
\nu(N,m) = \frac{1}{N} \sum_{k=N-m}^{N-1} \nu(k,m) & \ge \frac{1}{N} \sum_{k=N-m}^{N-1}  
\rho(k/m) > \frac{1}{N} \sum_{k=N-m}^{N-1} \int_{k}^{k+1} \rho(t/m)\, dt\\
&= \frac{1}{N} \int_{N-m}^N \rho(v/m)\, dv = \frac{1}{N/m}\int_{N/m-1}^{N/m}\rho(v)\, dv = \rho(N/m)
\end{align*}
and
\begin{align*}
\nu(N,m) \le \frac{1}{N} \sum_{k=N-m}^{N-1} \rho\pfrac{k+1}{m+1}
&\le \frac{1}{N} \sum_{k=N-m}^{N-1} \int_{k-1}^{k} \rho\pfrac{t+1}{m+1}\, dt \\
&= \frac{m+1}{N} \int_{\frac{N-m}{m+1}}^{\frac{N}{m+1}}\, \rho(v)\, dv\\
&= \frac{m+1}{N} \int_{\frac{N-m}{m+1}}^{\frac{N+1}{m+1}}\, \rho(v)\, dv
-  \frac{m+1}{N} \int_{\frac{N}{m+1}}^{\frac{N+1}{m+1}}\, \rho(v)\, dv\\
&= \frac{N+1}{N} \rho\pfrac{N+1}{m+1} - \frac{m+1}{N} \int_{\frac{N}{m+1}}^{\frac{N+1}{m+1}}\, \rho(v)\, dv.
\end{align*}
The final integral on the right side is $\ge \frac{1}{m+1} \rho\pfrac{N+1}{m+1}$
and thus $\nu(N,m) \le \rho\pfrac{N+1}{m+1}$.
The claimed bounds \eqref{nu-bounds} now follow by induction on $n$. 
\end{proof}


\section{Poisson approximation of small cycle lengths}\label{sec:Kubilius}


In this section, we prove Theorem \ref{Poisson_smallcycles}, which
 shows that $C_j(\sigma)$ is approximately Poisson with parameter
 $1/j$, uniformly for small $j$.

We begin by relating $d_{TV}(\cC_k, \cZ_k)$ to $\PR(C_{[m]}(\sigma)=0)$ using a
variant of a special case of 
\cite[eq. (33)]{AT94}.  Define $U(n,m) = \PR_n(C_{[m]}(\sigma)=0)$ for
$n\ge 0$ and $U(n,m)=0$ for $n<0$.

\begin{lem}\label{dTVCkZk}
We have
\[
d_{TV}(\cC_k, \cZ_k) = \sum_{\hh\in \NN_0^k} \; \prod_{j=1}^k \frac{(1/j)^{h_j}}{h_j!} \max \Big( 0, \er^{-H_k} -  U(n',k) \Big), 
\]
where $n'=n'(\hh) = n-\sum_{j=1}^k jh_j$.
\end{lem}

\begin{proof}
We begin with the easy identity
\[
d_{TV}(\cC_k, \cZ_k) = \sum_{\hh\in \NN_0^k} \max \Big( 0, \PR(\cZ_k = \hh) - \PR(\cC_k=\hh)  \Big). 
\]
Clearly,
\[
\PR(\cZ_k=\hh) = \er^{-H_k} \prod_{j=1}^k  \frac{(1/j)^{h_j}}{h_j!}.
\]
Now fix $\hh$, write $g = h_1+2h_2+\cdots+kh_k$ and consider $\PR(\cC_k=\hh)$.
If $g>n$, then $\PR(\cC_k=\hh)=0$.  Now suppose that $g\le n$.
Write $\sigma= \sigma_1 \sigma_2$, where $\sigma_1$ is the product of the cycles of length
at most $k$ and permutes a subset $I$ of $[n]$ of size 
$g$, and $\sigma_2$ is the product of the cycles of length greater than $k$ and permutes $[n]\setminus I$ of size $n'=n-g$.  By Cauchy's formula (Theorem \ref{thm:Cauchy}), applied to $\sigma_1$, it follows that
\[
\PR(\cC_k=\hh) = U(n',k) \prod_{j=1}^k  \frac{(1/j)^{h_j}}{h_j!},
\]
and the lemma follows.
\end{proof}

\begin{proof}[Proof of Theorem \ref{Poisson_smallcycles}]
We may assume that $k\le n/100$.  We will use Lemma \ref{dTVCkZk}
and estimate the contribution to $d_{TV}(\cC_k,\cZ_k)$
from the tuples
 $\hh=(h_1,\cdots,h_k)\in \NN_0^k$.  The main idea of the proof is to separately consider those
vectors which constitute rare events (many $h_j$ large): specifically, let
\begin{align*}
\mathcal{H}_1&=\{ \hh\in \NN_0^k: h_1+2h_2+\cdots+kh_k \le n-50k \}, \\
\mathcal{H}_2&=\{ \hh\in \NN_0^k: h_1+2h_2+\cdots+kh_k > n-50k \}.
\end{align*}
First, consider  $\hh\in \mathcal{H}_1$ and let $n'=n-(h_1+2h_2+\cdots+kh_k)\ge 50k$.
By Theorem \ref{nosmallcycles2},
\[
U(n',k) = \er^{-H_k} \(1 + O(\er^{-g(n'/k)})\),
\]
where $g(x)=-x\log x + x\log\log\log x+O(x)$ when $x\ge 50$. It follows that
\[
\sum_{\hh\in\mathcal{H}_1} \prod_{j=1}^k  \frac{(1/j)^{h_j}}{h_j!} \Big| \er^{-H_k} - U(n',k)\Big| \ll \er^{-H_k} \ssum{\hh\in \cH_1}  \er^{-g(n'/k)} \prod_{j=1}^k \frac{(1/j)^{h_j}}{h_j!}.
\]
For $\hh \in \mathcal{H}_2$, we use a trivial bound
\[
\max \Big( 0, \er^{-H_k} - U(n',k) \Big) \le \er^{-H_k} \le 1/k.
\]
We conclude that
\be\label{sumrhh}
\sum_{\hh\in\NN_0^k} \; \prod_{j=1}^k  \frac{(1/j)^{h_j}}{h_j!} \max \Big( 0, \er^{-H_k} - U(n',k)\Big) \ll \frac{1}{k}\sum_{50\le r\le n/k+1}\er^{-g(r)}
\ssum{\hh\in \NN_0^k \\ n' < rk} \;
\prod_{j=1}^k \frac{(1/j)^{h_j}}{h_j!}.
\ee
As in the proof of Theorem \ref{nolargecycles},
we invoke the method of parameters, also known as the tilting method
(this is commonly used in Chernoff inequalities; see Section 0.5 in \cite{Divisors} for number theoretic applications).
For any real number $w\ge 1$ we have
\begin{align*}
\ssum{\hh\in \NN_0^k \\ n'<rk} \; \prod_{j=1}^k \frac{(1/j)^{h_j}}{h_j!} &\le \sum_{\hh\in\NN_0^k}
w^{h_1+2h_2+\cdots+kh_k-n+rk} \prod_{j=1}^k \frac{(1/j)^{h_j}}{h_j!} 
\\ &= w^{-n+rk} \exp\left\{ w+ \frac12 w^2+ \cdots + \frac{1}{k} w^k\right\}.
\end{align*}
Take $w=(u-r+2)^{1/k}$ where $u=\frac{n}{k}$.  By the argument in \eqref{powers-w}, 
\[
 w+ \frac12 w^2+ \cdots + \frac{1}{k} w^k \le H_k+u-r+1
 \le \log k + u-r+2.
\]
It follows that
\[
\ssum{\hh\in \NN_0^k \\ n'<rk} \; \prod_{j=1}^k \frac{(1/j)^{h_j}}{h_j!}
\le k \exp \big\{ -(u-r)\log(u-r+2)+(u-r+2)\big\}.
\]
Inserting this into \eqref{sumrhh}, we find that
\begin{align*}
d_{TV}(\cC_k,\cZ_k) &\ll \er^{u\log\log\log u+O(u)} \sum_{50\le r\le u+1} \er^{-r\log r -(u-r)\log(u-r+2)}\\
&\ll \er^{u\log\log\log u+O(u)} \sum_{50\le r\le u+1} \frac{1}{r! (u+2-r)!}\\
&\ll \er^{-u\log u + u\log\log\log u+O(u)}.
\qedhere
\end{align*}
\end{proof}


\section{Central Limit Theorems}\label{sec:CLT}

A principal tool is the fact that, as $\lambda\to\infty$, 
the Poisson random variable with parameter $\lambda$ approaches a Gaussian distribution
with mean $\lambda$ and variance $\lambda$.  The following is a special case of the Central Limit Theorem with Berry-Esseen type rate of convergence.
For completeness, we give a short proof in the Appendix using only
 Stirling's formula and Euler summation.

\begin{lem}[Poisson CLT]\label{Poisson_CLT}
Let $\lambda \ge 1$, and let $X$ be Poisson with parameter $\lambda$.
Uniformly for real $\lambda\ge 1$ and real $z$, we have
\[
\PR \( X \le \lam + z\sqrt{\lam} \) = \Phi(z) + O\( \lambda^{-1/2} \), \qquad \Phi(z)=\frac{1}{\sqrt{2\pi}} \int_{-\infty}^z \er^{-\frac12 t^2}\, dt.
\]
\end{lem}

\begin{proof}[Proof of Theorem \ref{CLT_cycles}]
Let $H=H(I)$.  We may assume that $H\ge 100$, the assertion being trivial otherwise.
If $|w|\ge \sqrt{3\log H}$ then the result follows
from Theorem \ref{single-set-tails}, since
the left side is thus $O(1/H)=\Phi(w)+O(1/H)$ if $w\le -\sqrt{3\log H}$ and is $1-O(1/H)=\Phi(w)+O(1/H)$ if $w\ge \sqrt{3\log H}$.
Suppose now that $|w| < \sqrt{3\log H}$, let
\[
A = H + w \sqrt{H}, \qquad m=\cl{\frac{n}{\log H}}, \qquad J = I \cap [m].
\]
Because 
\[
H(I\setminus J) =\ssum{m<k\le n \\ k\in I} \frac{1}{k} \le H((m,n]\cap \NN) \le \log\log H+O(1)
\]
we have $H(J) = H + O(\log\log H)$.  Thus,
\[
A =  H(J) + w'\sqrt{H(J)}, \quad w'=w+O\pfrac{\log\log H}{\sqrt{H}}.
\]
Let $Y$ be a Poisson random variable with parameter $H(J)$.
Thus, by Theorem \ref{poisson_cycles} and Lemma \ref{Poisson_CLT},
\begin{align*}
\PR ( C_I(\sigma) \le A ) &\le \PR ( C_J (\sigma) \le A) \\
&= \PR(Y\le A) + O(\er^{-n/m}) \\
&=\Phi(w') + O\( H(J)^{-1/2} + \er^{-n/m} \) \\
&=\Phi(w') + O\pfrac{1}{\sqrt{H}} \\
&= \Phi(w) +O\pfrac{\log\log H}{\sqrt{H}}.
\end{align*}

 We also  have
\[
A - \log H = H(J) + w''\sqrt{H(J)}, \quad w''=w+O\pfrac{\log H}{\sqrt{H}}
\]
and it follows 
 that
\begin{align*}
\PR ( C_I(\sigma) \le A ) &\ge \PR \( C_{J} (\sigma) \le A - \log H  \text{ and } C_{I\setminus J}(\sigma) \le \log H  \)\\
&=\PR \( C_{J} (\sigma) \le A - \log H  \),
\end{align*}
since $\min (I\setminus J) \ge n/\log H$ implies that $C_{I\setminus J}(\sigma)\le \log H$ always.  Hence, by Theorem \ref{poisson_cycles} and Lemma \ref{Poisson_CLT},
\begin{align*}
\PR ( C_I(\sigma) \le A )& \ge \Phi(w'') + O(1/\sqrt{H}) \\
&=\Phi(w) + O\pfrac{\log H}{\sqrt{H}}.
\end{align*}
The theorem follows by combining the upper and lower bounds
for $\PR(C_I(\sigma)\le A)$.
\end{proof}

\begin{proof}[Proof of Theorem \ref{jth-smallest-cycle}]
We may assume that $j\ge 10$ and that $n$ is sufficiently large, the statement being trivial otherwise.  We may also assume that $|w| \le \sqrt{\log j}$, since the
statement for $w$ outside this range follows from the
monotonicity of $\PR(\log D_j(\sigma) \le j+w\sqrt{j})$, as a function of $w$,
 the statement for the two points $w = \pm \sqrt{\log j}$ and the fact that 
 $\Phi(-\sqrt{\log j})\ll 1/j^{1/2}$ and $\Phi(\sqrt{\log j})=1-O(1/j^{1/2})$.

Let $k=\fl{\er^{j+w\sqrt{j}}}$, so by hypothesis,
\[
\log k \le j + \sqrt{j\log j} \le j + \sqrt{(\log n)\log\log n} \le \log n.
\]
  Then $D_j(\sigma)\le k$ is equivalent to $C_{[k]}(\sigma) \ge j$. 
 As $H_k=\log k + O(1)$ and $\sqrt{H_k}=\sqrt{j}+O(|w|+1)$, we have
\[
j - 1 = H_k - u\sqrt{H_k}, \quad\text{where}\quad u = w + O\pfrac{w^2+1}{\sqrt{j}}.
\]
 By Theorem \ref{CLT_cycles},
\begin{align*}
\PR(D_j(\sigma)\le k) &= \PR(C_{[k]}(\sigma) \ge j) = 1 - \PR(C_{[k]}(\sigma) \le j-1)\\
&= 1-\Phi(u)+O \( \frac{\log H_k}{\sqrt{H_k}} \)\\
&= \Phi(u) + O \pfrac{\log(2j)}{\sqrt{j}}.
\end{align*}
 Also,
\[
\Phi(u) = \Phi(w)+O\pfrac{w^2+1}{\sqrt{j}} = \Phi(w) + 
O\pfrac{\log (2j)}{\sqrt{j}}
\]
and the proof is complete.
\end{proof}


\section{Fixed sets  and divisors of permutations}

\begin{proof}[Proof of Theorem \ref{exp-numdiv}]
Evidently, $2^{C(\sigma)}$ equals the number of 
divisors $\beta|\sigma$.  The permutation $\beta$
fixes a set $I$.  Summing over $I$ we see that
\begin{align*}
\E 2^{C(\sigma)} &= \frac{1}{n!} \sum_{\sigma\in \cS_n} \sum_{\beta|\sigma}
1 = \frac{1}{n!} \sum_{I \subseteq [n]} \ssum{\sigma\in \cS_n \\ \sigma\text{ fixes
  }I} 1 \\
&= \frac{1}{n!}  \sum_{I \subseteq [n]} (n-|I|)! |I|! \\
&= \frac{1}{n!} \sum_{j=0}^n (n-j)!j!  \binom{n}{j} = \sum_{j=0}^n 1 = n+1. \qedhere
\end{align*}
\end{proof}

\begin{proof}[Proof of Theorem \ref{thm:fixed-set}]
The statement is trivial for $1\le k\le 100$, thus we may assume
that $k>100$.  Let $r_0 = \frac{H_k}{\log 2}$, so that
$r_0 = \frac{\log k}{\log 2} + O(1)$. 
 By Theorem \ref{cyclesinterval},
\[
\PR( C_{[k]}(\sigma) \ge r_0 ) \ll k^{-Q(1/\log 2)}  = k^{-\cE}.
\]
If $\sigma$ has a fixed set of size $k$, then $\sigma$ factors
as $\sigma = \a \b$, where $|\a| = k$ and $|\b|=n-k$.
Hence, if $C_{[k]}(\sigma)< r_0$, then for some
non-negative integers $j,h$ with $j+h<r_0$ we have
\be\label{alphabeta}
C(\alpha)=j, \qquad C_{[k]}(\b)=h.
\ee
With $j,h$ fixed the number of pairs $\a,\b$ with
\eqref{alphabeta} is at most
\[
\binom{n}{k} k! \PR_k(C(\alpha)=j) (n-k)! \; \PR_{n-k}(C_{[k]}(\b)=h)
\ll n! \frac{H_k^{j+h} \er^{-2H_k}}{j!h!},
\]
upon invoking Lemma \ref{cycles_sets}.
Summing first over all $j,h$ with $h+j=r$ using the binomial theorem, and then over $r < r_0$ we see that the
probability that $C_{[k]}(\sigma) < r_0$ and
 $\sigma$ factors as $\sigma=\a\b$ with
$|\a|=k$ is bounded above by
\[
\ll \er^{-2H_k} \sum_{r< r_0} \frac{(2H_k)^r}{r!} \ll
k^{-Q(\frac{1}{2\log 2})} = k^{-\cE},
\]
upon invoking Lemma \ref{Poisson_tails}.
\end{proof}

\appendix 
\section*{Appendix}

In this appendix, we proof Lemma \ref{Poisson_CLT} and \eqref{rho-local}.
\begin{proof}[Proof of Lemma \ref{Poisson_CLT}]
We give a short, direct proof using Stirling's formula and Euler summation.
Let $h^*=3\sqrt{\log(1+\lambda)}$.  We may assume that $\lambda$ is sufficiently large.
By Proposition \ref{Poisson_tails} and the crude bounds for $Q(x)$ given in \eqref{Qx_crude}, we have
\[
\PR(|X-\lam| > h^*\sqrt{\lambda}) \le 2\er^{-3\log (1+\lambda)} = \frac{2}{(1+\lambda)^3}.
\]
Likewise,
\be\label{Gauss_tail}
\int\limits_{|t|>h^*} \er^{-\frac12 t^2}\, dt \ll \frac{1}{(1+\lambda)^3}.
\ee
Consequently, we may assume that $|z| \le h^*$, and deduce 
\[
\PR \( X \le \lam + z\sqrt{\lam} \) = \er^{-\lam} \sum_{\lam-h^*\sqrt{\lam} \le k\le \lam+z\sqrt{\lam}} \frac{\lam^k}{k!} + O\pfrac{1}{\lam^3}.
\]
For $|k-\lam| \le h^*\sqrt{\lam}$, Stirling's formula implies that
\[
k! = \pfrac{k}{\er}^k \sqrt{2\pi \lam} \(1 + O\pfrac{|k-\lam|+1}{\lam} \).
\]
Write $k=\lam+u$.  Then, for $|u| \le h^*\sqrt{\lam}$, we have
\begin{align*}
\er^{-\lam} \frac{\lam^k}{k!} &=  \frac{1+O\pfrac{|u|+1}{\lam}}{\sqrt{2\pi\lam}} \er^{-\lam}
\pfrac{\er\lam}{\lam+u}^{\lam+u} = \frac{1+O\pfrac{|u|+1}{\lam}}{\sqrt{2\pi\lam}} \frac{\er^u}{(1+u/\lam)^{\lam+u}} \\
&=\frac{1+O\pfrac{|u|+1}{\lam}} {\sqrt{2\pi\lam}}
\exp\left\{ u-(\lam+u)\(\frac{u}{\lam}-\frac12 \pfrac{u}{\lam}^2
+O\(\pfrac{u}{\lam}^3\) \) \right\} \\
&=\(1+ O\(\frac{1+|u|}{\lam} + \frac{|u|^3}{\lam^2}\)\) 
\frac{\er^{-\frac{u^2}{2\lam}}}{\sqrt{2\pi \lam}}.
\end{align*}
It follows that
\[
\er^{-\lam} \sum_{\lam-h^*\sqrt{\lam} \le k\le \lam+z\sqrt{\lam}} \frac{\lam^k}{k!} = M+E,
\]
where
\[
M = \frac1{\sqrt{2\pi \lam}} \sum_{\lam-h^*\sqrt{\lam} \le k\le \lam+z\sqrt{\lam}} \er^{-\frac{(k-\lam)^2}{2\lam}}
\]
and 
\begin{align*}
E &\ll \frac{1}{\sqrt{\lam}} \sum_k \( \frac{1+|k-\lam|}{\lam} + \frac{|k-\lam|^3}{\lam^2} \) \er^{-\frac{|k-\lam|^2}{2\lam}} \\
&\ll \sum_{a=1}^\infty \( \frac{a+a^3}{\sqrt{\lam}} \)
\er^{-(a-1)^2/2} \ll \frac{1}{\sqrt{\lam}}. 
\end{align*}
By Euler summation, and writing $\{t\}=t-\fl{t}$, 
\[
M=\frac{1}{\sqrt{2\pi \lam}} \Bigg[ \int_{\lam-h^*\sqrt{\lam}}^{\lam+z\sqrt{\lam}} \er^{-\frac{(t-\lam)^2}{2\lam}}\, dt - \int_{\lam-h^*\sqrt{\lam}}^{\lam+z\sqrt{\lam}} \{t \} \pfrac{t-\lam}{\lam} \er^{-\frac{(t-\lam)^2}{2\lam}}\, dt +O(1) \Bigg]. 
\]
The integral involving $\{t\}$ is $O(1)$.
The first integral equals, by \eqref{Gauss_tail},
\[
\sqrt{\lam} \int_{-h^*}^z \er^{-\frac12 u^2}\, du = \sqrt{\lam} \int_{-\infty}^z \er^{-\frac12 u^2}\, du + O(\lam^{-5/2}),
\]
and hence
\[
M = \frac{1}{\sqrt{2\pi}} \int_{-\infty}^z \er^{-\frac12 u^2}\, du + O\pfrac{1}{\sqrt{\lam}} = \Phi(z) + O\pfrac{1}{\sqrt{\lam}}.
\qedhere
\]
\end{proof}

\begin{proof}[Proof of \eqref{rho-local}]
It suffices to show that
\be\label{rho-d}
- \frac{\rho'(u)}{\rho(u)} \ll 1 + \log u \qquad (u > 1).
\ee
From \eqref{rho-recur} and \eqref{rho-int},
\be\label{rho-logderiv}
-\frac{\rho'(u)}{\rho(u)} = \frac{\rho(u-1)}{\int_{u-1}^u \rho(v)\, dv}.
\ee
 Let $B_k = \max_{1<v\le k/2} (-\rho'(v)/\rho(v))$.  We have
\[
B_4 = \max_{1<v\le 2} \frac{1/v}{1-\log v} = \frac{1}{2(1-\log 2)} = 1.629\ldots.
\]
If $k\ge 4$ and $k/2 < u \le (k+1)/2$ then the denominator on the right side of \eqref{rho-logderiv} is
at least
\[
\int_{u-1}^{u-1/2} \rho(v)\, dv \ge \rho(u-1) \int_{u-1}^{u-1/2} \er^{-B_k(v-u+1)}\, dv
= \frac{\rho(u-1) (1-\er^{-\frac12 B_k})}{B_k}. 
\]
Using that $\er^{-\frac12 B_k} \le \er^{-\frac12 B_4} < 1/2$,
we infer that
\[
B_{k+1} \le \frac{B_k}{1-\er^{-\frac12 B_k}} \le B_k \(1 + 2 \er^{-\frac12 B_k}\).
\]
The function $x(1+2\er^{-x/2})$ is increasing for $x\ge 0$, hence if $C$ is large and
$B_k \le C\log k$ then \[
B_{k+1} \le (C\log k)(1+2/k^{C/2}) \le C\log(k+1).
\]
Therefore, $B_k \ll \log k$ and \eqref{rho-d} follows.
\end{proof}

Somewhat stronger local bounds on $\rho(u)$, also proved by elementary
methods, can be found in section 2 of \cite{hildebrand}.

\section*{Acknowledgments} 
The author thanks Sean Eberhard and Ben Green for helpful comments on
an early draft, and thanks Dimitris Koukoulopoulos for showing him the lower
bound argument in Theorem \ref{nu-asym}.
  The author also thanks the  anonymous referee for carefully reading the paper
  and making many helpful suggestions.

\bibliographystyle{amsplain}

\begin{dajauthors}
\begin{authorinfo}[kbf]
  Kevin Ford\\
  Department of Mathematics\\
  University of Illinois at Urbana-Champaign\\
  1409 West Green Street\\
  Urbana, IL 61801, USA\\
  ford\imageat{}math\imagedot{}uiuc\imagedot{}edu\\
  \url{https://faculty.math.illinois.edu/~ford/}
\end{authorinfo}
\end{dajauthors}

\end{document}